\documentclass{article} 
\usepackage{iclr2027_conference,times}

\usepackage{amsmath,amsfonts,bm}

\def\eqref#1{equation~\ref{#1}}

\def\1{\bm{1}}

\DeclareMathAlphabet{\mathsfit}{\encodingdefault}{\sfdefault}{m}{sl}
\SetMathAlphabet{\mathsfit}{bold}{\encodingdefault}{\sfdefault}{bx}{n}

\DeclareMathOperator*{\argmin}{arg\,min}

\usepackage{amsmath, amssymb, amsthm, mathtools}
\usepackage{bbm}
\usepackage{enumitem}
\usepackage[colorlinks=true, citecolor=blue, linkcolor=blue, urlcolor=blue]{hyperref}
\usepackage[nameinlink]{cleveref}

\newtheorem{theorem}{Theorem}[section]
\newtheorem{proposition}[theorem]{Proposition}
\newtheorem{lemma}[theorem]{Lemma}

\usepackage{hyperref}
\usepackage{url}

\title{A statistical theory for blind denoising: \\minimax estimation of the noise level \\from a single sample}

\author{
\textbf{Zhuoer Shen}\\
Yale University\\
\texttt{zhuoer.shen@yale.edu}
}

\iclrfinalcopy 

\begin{document}

\maketitle

\begin{abstract}
Motivated by blind denoising in diffusion models, we study estimation of an
unknown Gaussian noise level from a single high-dimensional observation,
assuming the signal law \(P\) is known. We characterize the minimax
mean-squared error under two structural assumptions on \(P\).
For signals with covering complexity \(k\), the minimax rate is
\(\widetilde{\Theta}_\Lambda
(\min\{\Delta_\Lambda^2,d^{-1}+k^2d^{-2}\})\),
and the constrained MLE attains it up to logarithmic factors.
For \(\alpha\)-strongly log-concave signals, the rate is
\(\Theta_\Lambda
(\min\{\Delta_\Lambda^2,(1+\alpha^{-1})^2d^{-1}\})\),
attained up to constants by a \(P\)-centered norm estimator.
These results show that the structure of the signal law determines both
the difficulty of blind noise estimation and the appropriate estimator.
\end{abstract}

\section{Introduction}
\label{sec:introduction}

Diffusion models generate high-dimensional data by reversing a gradual corruption process through a denoiser or score field indexed by diffusion time
\citep{sohldickstein2015deep,ho2020denoising,song2021scorebased}.
They have also proved effective for conditional generation
\citep{dhariwal2021diffusion}. A closely related class of generative models is flow matching, which learns a time-dependent velocity field along a probability path connecting a simple base distribution to the data distribution
\citep{lipman2023flow,albergo2023building,liu2023flow}. In both cases, the model typically uses a network that takes time as an input, or the corresponding noise level along a Gaussian path
\citep{song2019generative,ho2020denoising,lipman2023flow}.

Recent work has asked whether the network needs this explicit time input.
Noise-unconditioned denoising models can remain competitive with their
conditioned counterparts \citep{sun2025noiseconditioning}, while blind
denoising diffusion models omit the noise amplitude during both training and
sampling \citep{kadkhodaie2026blinddenoising}. Relatedly, in flow matching,
Beckmann Transport Models show that, under suitable geometric conditions, an
autonomous velocity field defines an exact transport and induces a one-step
generative map
\citep{lee2026beckmann}. Taken together, these results suggest that explicit noise conditioning is not always necessary. What remains less understood is why
a model can recover the missing noise information from the noisy state itself.
This leads to a basic statistical question: if the noise level is not provided,
what properties of the signal distribution make it recoverable from a single
noisy observation?

Concretely, a blind denoiser must infer which noisy marginal produced its current input. This problem is closely related to classical noise-level estimation: estimating unknown noise parameters from a single
corrupted observation is a classical difficulty in blind image denoising
\citep{liu2013single}, and recent diffusion-based approaches likewise treat
unknown noise parameters as an additional inference problem
\citep{heurteldepeiges2024listening}.
More recently, \citet{kadkhodaie2026blinddenoising} formalized this problem in their analysis
of blind denoising diffusion models (BDDMs). Under a low-intrinsic-dimensionality assumption, quantified by the covering
complexity of the signal support at the relevant noise scale, they showed that
the posterior over the noise level concentrates near its true value, allowing
the blind score to approximate the corresponding noise-conditioned score and
the model to track an implicit noise schedule. Their analysis yields an
estimation upper bound at the scale $d^{-1}+k^2d^{-2}$, where $d$ is the ambient dimension and $k$ is the corresponding
intrinsic dimension. This bound is sufficient to control the additional error introduced by blind
sampling, but does not determine whether the estimation scale is statistically
optimal or whether low intrinsic dimensionality is necessary for accurate
estimation.

In this work, we study this estimation problem separately from the diffusion dynamics. Consider
\[
    Y=X+\lambda^{-1/2}Z,
    \qquad X\sim P,\quad Z\sim\mathcal{N}(0,I_d),
\]
where the signal law \(P\) is known, the unknown noise level is parameterized
by its precision \(\lambda\in\Lambda\), and we observe a single
\(d\)-dimensional vector \(Y\). The estimator may use the full law \(P\); this oracle formulation excludes
distribution-learning, neural-network approximation, score-estimation, and
discretization errors.

The ambient dimension might appear to provide \(d\) effective observations
and hence a parametric rate. The signal coordinates,
however, need not be independent, and variation in \(X\) can imitate a change
in the noise level. Consequently, high dimension alone does not guarantee accurate estimation: it becomes useful only when the structure of \(P\) prevents
signal variation from concealing noise variation. We therefore ask:
\[
    \textit{Which structural properties make blind noise estimation possible,
    and at what optimal rate?}
\]

\subsection{Our Contributions}
\label{sec:intro-contributions}

Our results show that accurate blind noise estimation is possible under more than one structural regime. We characterize the minimax difficulty in these regimes and give estimators attaining the corresponding rates.

\paragraph{Minimax rates.}
Let $d$ denote the ambient dimension, $k$ the intrinsic dimension
in the low-dimensional regime, $\alpha$ the strong log-concavity parameter,
and $\Delta_\Lambda$ the width of the fixed precision interval. Under the
assumptions stated below, and up to constants depending on $\Lambda$, the
minimax mean-squared error has order
\[
\widetilde{\Theta}_\Lambda\!\left(
\min\left\{\Delta_\Lambda^2,\frac{1}{d}+\frac{k^2}{d^2}\right\}
\right)
\]
under low intrinsic dimensionality, and
\[
\Theta_\Lambda\!\left(
\min\left\{\Delta_\Lambda^2,\frac{(1+\alpha^{-1})^2}{d}\right\}
\right)
\]
under strong log-concavity. The formal statements appear in Theorems~\ref{thm:ld-main}
and~\ref{thm:slc}.

Under low intrinsic dimensionality, we prove that the $d^{-1}$ term is a
parametric lower bound for every fixed signal law $P$, while the $k^2/d^2$
term arises from latent signal variation that can conceal a precision
perturbation of order $k/d$. The constrained maximum-likelihood estimator
(MLE) matches these lower bounds up to logarithmic factors, showing that the
scale obtained in the blind-denoising analysis of
\citet{kadkhodaie2026blinddenoising} is essentially
unimprovable in the fixed-$P$ setting.

Under strong log-concavity, low-dimensional support is not necessary.
For $\alpha$-strongly log-concave signal laws, which may have full-dimensional
support, we construct a $P$-centered norm estimator that attains the minimax
rate up to constants. The MLE has the optimal dependence on $\alpha$ with an
additional logarithmic factor in dimension. Thus, accurate noise estimation does not require low-dimensional support.

\paragraph{Norm-based and likelihood-based estimation.}
Strong log-concavity controls the radial fluctuations needed for the norm
estimator through a Poincar\'e inequality. In contrast, low covering
complexity alone does not control $\operatorname{Var}_P(\|X\|^2)$, so the
standard energy-corrected norm statistic has no uniform variance bound in
terms of $k$ alone. This explains why the same simple statistic need not
attain the low-dimensional rate, whereas the MLE adapts to the full signal law. We develop this comparison in Section~\ref{sec:upper-techniques} and
Appendix~\ref{app:ld-proofs}.

\paragraph{Proof ideas.}
We develop separate lower- and upper-bound arguments for the two structural
regimes.
For the $k^2/d^2$ lower bound, we construct a single variance-mixture signal
law and approximate it by a finite point cloud, preserving overlap between
noisy marginals while enforcing the covering condition. For the MLE, we use
two localization arguments. Under low intrinsic dimensionality,
posterior-energy control and monotonicity yield a quadratic likelihood
comparison. Under strong log-concavity, we instead establish R\'enyi
separation and combine it with a one-dimensional covering argument. The separation-based argument isolates the structural input in a bound between
the noisy marginals, suggesting a route to other signal classes for which
comparable separation estimates can be established. Section~\ref{sec:techniques} gives the proof overview, with complete arguments
in Appendices~\ref{app:ld-proofs} and~\ref{app:slc-proofs}.

\subsection{Prior Work}
\label{sec:intro-prior-work}

\paragraph{Blind denoising and noise-unconditioned generative models.}
Classical single-image noise estimators infer the noise level from low-texture or low-rank image patches
\citep{liu2013single,chen2015efficient}.
Learning-based denoisers have treated the role of the noise level in different ways.
DnCNN demonstrated blind Gaussian denoising with an unknown noise level
\citep{zhang2017beyond}, whereas FFDNet takes a noise-level map as an explicit input while using a single network across a range of noise levels
\citep{zhang2018ffdnet}. \citet{gnanasambandam2020one} instead study how to choose the training distribution over noise levels when one denoiser is used across the entire range. Blind denoisers have also been used
as implicit priors for inverse problems \citep{kadkhodaie2021stochastic}.
More recently, \citet{sun2025noiseconditioning} analyzed the error incurred by
removing noise conditioning from denoising generative models, while
\citet{heurteldepeiges2024listening} proposed a Gibbs procedure that jointly
samples the signal and unknown parametric noise variables using a diffusion
prior. \citet{kadkhodaie2026blinddenoising} is the closest work to ours: their BDDM analysis gives the low-dimensional noise-estimation upper bound, whereas we study the minimax limits of this estimation problem.

\paragraph{Diffusion models under low-dimensional structure.}
Empirical evidence suggests that the intrinsic dimension of natural images is
much smaller than their ambient dimension \citep{pope2021intrinsic}. Theoretical
work shows that score-based diffusion models can adapt their sampling
or discretization complexity to such structure
\citep{huang2024denoising,li2024adapting,
potaptchik2025linear,liang2025lowdimensional}. These results concern sampling
with supplied noise levels and accurate scores. In the blind setting,
low-dimensional structure plays a different role: it can make the missing
noise level recoverable from the current noisy state
\citep{kadkhodaie2026blinddenoising}. Our low-dimensional result characterizes
the optimal estimation scale associated with this mechanism.

\paragraph{Theoretical analyses of diffusion models.}
Many analyses condition on an accurate score estimator to separate
score-estimation error from initialization and discretization errors
\citep{chen2023improved,chen2023sampling,lee2023convergence}. Complementary
work studies the statistical problem of learning the score itself from finite
data \citep{wibisono2024optimal}. Our formulation isolates noise-level estimation with the signal law given. Strong log-concavity is also a standard structural assumption in the theory
of high-dimensional sampling and diffusion models, where it yields
concentration and functional inequalities that support quantitative
convergence guarantees
\citep{saumard2014logconcavity}.

\paragraph{Convolution models with unknown noise.}
Unknown Gaussian noise variance has also been studied in semiparametric
convolution models \citep{matias2002semiparametric,butucea2005minimax}. In that
literature, the signal density is unknown, information accumulates through
multiple independent observations, and smoothness assumptions ensure
identifiability. Here, one noisy vector is observed with \(P\) given, and the rates depend
on its covering complexity \(k\) or strong log-concavity parameter \(\alpha\).

\section{Background}
\label{sec:background}

\subsection{From Blind Denoising to Estimation of the Noise Level}
\label{sec:blind-to-estimation}

Additive Gaussian noise is a standard model of corruption in denoising and
diffusion models~\citep{sohldickstein2015deep,song2021scorebased}. For a
clean signal \(X\sim P\) and an independent \(Z\sim\mathcal{N}(0,I_d)\), a
noisy observation at level \(\sigma>0\) has the form
\[
    Y_\sigma=X+\sigma Z,
    \qquad
    P_\sigma:=P*\mathcal{N}(0,\sigma^2I_d).
\]
Standard diffusion models typically use denoisers that take the noise level as
an explicit input, leading to somewhat ad hoc choices of architecture and
diffusion schedule. Recently, \citet{kadkhodaie2021stochastic}
showed that one can instead use a \emph{blind} denoiser, trained across
noise levels \(\sigma\sim\Pi_0\) without receiving \(\sigma\) itself as an input.
Given a network \(f_\theta:\mathbb{R}^d\to\mathbb{R}^d\), the training
objective is
\[
    \mathbb{E}_{X,\sigma,Z}
    \bigl[\|X-f_\theta(X+\sigma Z)\|^2\bigr],
    \qquad
    \sigma\sim\Pi_0,\quad Z\sim\mathcal{N}(0,I_d).
\]
Let \(f_\sigma^\star(y):=\mathbb{E}[X\mid Y_\sigma=y]\) denote the
noise-conditioned Bayes denoiser. The population-optimal blind denoiser is
\[
    f^\star(y)
    =
    \int f_\sigma^\star(y)\,\Pi(d\sigma\mid y),
    \qquad
    \Pi(d\sigma\mid y)
    =
    \frac{p_\sigma(y)\Pi_0(d\sigma)}
         {\int p_\omega(y)\Pi_0(d\omega)}.
\]

Surprisingly, blind denoising can even improve image generation quality
despite the missing noise input~\citep{kadkhodaie2026blinddenoising}.
This work gave a theoretical explanation under
low intrinsic dimension: the posterior
\(\Pi(\cdot\mid Y)\) concentrates near the true noise level
\(\sigma_\star\), so the blind Bayes denoiser is close to the corresponding
noise-conditioned denoiser \(f_{\sigma_\star}^\star(Y)\). The same noisy
marginal determines the latter through the Tweedie--Miyasawa identity
\[
    f_\sigma^\star(y)-y
    =
    \sigma^2\nabla\log p_\sigma(y),
\]
which links denoising to the score of
\(P_\sigma\)~\citep{miyasawa1961empirical,efron2011tweedie}.

In this work, we study blind estimation of the noise level as a statistical
problem in its own right. In particular, we ask whether the statistical rate
obtained under low intrinsic dimension is tight, and under what other
conditions blind noise estimation is possible.

The posterior \(\Pi(\cdot\mid y)\) combines the prior weights in \(\Pi_0\)
with likelihood ratios between the marginals
\(\{P_\sigma\}_{\sigma>0}\). If two noise levels induce nearly
indistinguishable laws for \(Y\), then they cannot be reliably
distinguished from the current state. We therefore study the statistical
problem of estimating the noise level from \(Y\) for a fixed signal law
\(P\), independently of the choice of training distribution \(\Pi_0\).

\subsection{Statistical formulation and likelihood identities}
\label{sec:statistical-formulation}

We parameterize the noise by its precision
\(\lambda:=\sigma^{-2}\), and fix
\(\Lambda=[\lambda_{\min},\lambda_{\max}]\) with
\(0<\lambda_{\min}<\lambda_{\max}<\infty\).
Since \(\Lambda\) is a fixed compact interval, estimating \(\lambda\) or
\(\sigma=\lambda^{-1/2}\) under squared-error loss gives equivalent rates,
up to constants depending only on \(\Lambda\). We work with \(\lambda\) because
the likelihood takes a simpler form:
\[
    Y=X+\lambda^{-1/2}Z,
    \qquad
    P_\lambda:=P*\mathcal{N}(0,\lambda^{-1}I_d),
    \qquad
    p_\lambda(y)
    =
    \left(\frac{\lambda}{2\pi}\right)^{d/2}
    \int
    e^{-\frac{\lambda}{2}\|x-y\|^2}\,P(dx).
\]
Write
\(\Delta_\Lambda:=\lambda_{\max}-\lambda_{\min}\).

For each fixed \(P\), define
\[
    \mathcal{R}(P;\Lambda)
    :=
    \inf_{\widehat{\lambda}_P}
    \sup_{\lambda\in\Lambda}
    \mathbb{E}_{\lambda,P}
    \left[
        \bigl(\widehat{\lambda}_P(Y)-\lambda\bigr)^2
    \right],
\]
where the infimum is over estimators
\(\widehat{\lambda}_P:\mathbb{R}^d\to\Lambda\).
Throughout, \(P\) is treated as known, and
\(\widehat{\lambda}_P\) may depend arbitrarily on \(P\), including through
its full likelihood. In Section~\ref{sec:results}, we take the supremum of
\(\mathcal{R}(P;\Lambda)\) over each structural class.

Let \(\ell(\lambda\mid y):=-\log p_\lambda(y)\), and define the posterior
quadratic error
\[
    D_\lambda(y)
    :=
    \mathbb{E}_{\lambda,P}
    \left[\|X-y\|^2\mid Y=y\right].
\]
Direct differentiation gives
\[
    2\ell'(\lambda\mid y)
    =
    -\frac{d}{\lambda}+D_\lambda(y),
    \qquad
    \partial_\lambda D_\lambda(y)
    =
    -\frac12
    \operatorname{Var}_{\lambda,P}
    \!\left(\|X-y\|^2\mid Y=y\right)
    \le 0.
\]

With \(\lambda=\sigma^{-2}\), \(D_\lambda(Y)\) is the conditional
quadratic error appearing in the blind-denoising analysis.
We consider the constrained MLE
\[
    \widehat{\lambda}_{\mathrm{MLE},P}
    \in
    \arg\min_{\lambda\in\Lambda}\ell(\lambda\mid Y).
\]
Its localization is analyzed in Section~\ref{sec:upper-techniques};
the likelihood derivations and auxiliary inequalities are collected in
Appendix~\ref{app:preliminaries}.
\section{Results}
\label{sec:results}

We now formally state our assumptions and our main results.

\subsection{Results for low intrinsic dimensionality}
\label{sec:ld-results}

We first consider signal distributions with low geometric complexity. Define the resolution
\[
r_0
:=
\frac{\sqrt{\lambda_{\min}}}
{\lambda_{\max}\sqrt d}.
\]
For $k\geq 1$ and $R>0$, let
$\mathcal{P}_{\mathrm{LD}}(k,R)$ be the class of probability measures
$P$ on $\mathbb{R}^d$ satisfying
\[
\operatorname{supp}(P)\subseteq B(0,R),
\qquad
1+\log N(\operatorname{supp}(P),r_0)\leq k,
\]
where $N(S,r_0)$ denotes the Euclidean covering number of $S$ at
scale $r_0$. The support may be discrete or highly nonconvex. Its covering complexity at
this resolution is bounded by $k$, which need not scale with the ambient
dimension $d$.

Using the risk $\mathcal{R}(P;\Lambda)$ defined in
Section~\ref{sec:background}, set
\[
\mathcal{R}^{\star}_{\mathrm{LD}}(d,k,R;\Lambda)
:=
\sup_{P\in\mathcal{P}_{\mathrm{LD}}(k,R)}
\mathcal{R}(P;\Lambda).
\]

Our main result shows that, up to logarithmic factors, the minimax risk is
governed by the two scales \(d^{-1}\) and \(k^2d^{-2}\).

\begin{theorem}[Low intrinsic dimensionality]
\label{thm:ld-main}
Assume $1\leq k\lesssim_{\Lambda} d$ and $R^2\gtrsim_{\Lambda} k$. Then
\[
\mathcal{R}^{\star}_{\mathrm{LD}}(d,k,R;\Lambda)
\gtrsim
\min\left\{
\Delta_{\Lambda}^2,\,
\lambda_{\min}^2
\left(
\frac{1}{d}+\frac{k^2}{d^2}
\right)
\right\}.
\]
Conversely, for every fixed
$P\in\mathcal P_{\mathrm{LD}}(k,R)$, the MLE defined in
Section~\ref{sec:statistical-formulation} satisfies
\[
\sup_{\lambda^\star\in\Lambda}
\mathbb{E}_{\lambda^\star,P}
\bigl[
(\widehat{\lambda}_{\mathrm{MLE},P}-\lambda^\star)^2
\bigr]
\lesssim
\min\left\{
\Delta_{\Lambda}^2,\,
\frac{\lambda_{\max}^4}{\lambda_{\min}^2}
\left[
\frac{\log(2d)}{d}
+
\frac{(k+\log(2d))^2}{d^2}
\right]
\right\}.
\]
Hence, for a fixed compact precision interval \(\Lambda\),
\[
\mathcal{R}_{\mathrm{LD}}^\star(d,k,R;\Lambda)
=
\widetilde{\Theta}_{\Lambda}
\left(
\min\left\{
\Delta_\Lambda^2,\,
\frac{1}{d}+\frac{k^2}{d^2}
\right\}
\right).
\]
\end{theorem}

The two terms dominate in different regimes. When $k \lesssim \sqrt{d}$, the parametric term $d^{-1}$ determines the rate, whereas for $k \gtrsim \sqrt{d}$, the structural term $k^2/d^2$ dominates. The radius $R$ enters through feasibility of the structural class and the lower-bound construction; once $R^2 \gtrsim_{\Lambda} k$, it does not determine the resulting minimax scale. This shows that the low-dimensional estimation rate obtained by
\citet{kadkhodaie2026blinddenoising} is tight up to logarithmic factors.

\paragraph{Proof sketch.}
The \(d^{-1}\) lower bound follows from a parametric two-point argument
for fixed \(P\). The \(k^2/d^2\) term uses a variance-mixture construction
with precision separation of order \(k/d\); see
Section~\ref{sec:lower-techniques} and Appendix~\ref{app:ld-proofs}.

For the upper bound, control of $D_\lambda(Y)$ together with the likelihood identity in Section~\ref{sec:background} compares the observed negative log-likelihood with a quadratic Gaussian reference. This yields a localization bound for the MLE; see Section~\ref{sec:upper-techniques} and Appendix~\ref{app:ld-proofs}.

\subsection{Results for Strong Log-Concavity}
\label{sec:slc-results}

Next, we investigate whether we can go beyond low intrinsic dimensionality by identifying another class of structured signal priors for which blind noise estimation is possible. We consider strongly log-concave signal distributions. For $\alpha>0$, let
\[
\mathcal P_{\mathrm{SLC}}(\alpha)
:=
\left\{
P(dx)=Z_V^{-1}e^{-V(x)}dx:
V\in C^2(\mathbb R^d),\ 
\nabla^2V(x)\succeq \alpha I_d
\right\}.
\]
For $P\in\mathcal P_{\mathrm{SLC}}(\alpha)$, let
$\mu_P:=\mathbb E_PX$. Since $P$ is known, so is $\mu_P$. Translating $(X,Y)$ to $(X-\mu_P,Y-\mu_P)$ preserves both the noise
precision and strong log-concavity. Unlike the class in Section~\ref{sec:ld-results}, this condition controls
signal spread through $\alpha$ without restricting support dimension.

Using the risk $\mathcal{R}(P;\Lambda)$ defined in
Section~\ref{sec:background}, set
\[
\mathcal{R}^\star_{\mathrm{SLC}}(d,\alpha;\Lambda)
:=
\sup_{P\in\mathcal P_{\mathrm{SLC}}(\alpha)}
\mathcal{R}(P;\Lambda),
\qquad
\rho_{\alpha,d}(\Lambda)
:=
\min\left\{
\Delta_\Lambda^2,\,
\frac{(1+\alpha^{-1})^2}{d}
\right\}.
\]

Our main result shows that, up to constants depending on \(\Lambda\), the minimax risk is governed by the scale \((1+\alpha^{-1})^2/d\).

\begin{theorem}[Strong log-concavity]
\label{thm:slc}
For every $d\ge 1$ and $\alpha>0$,
\[
\mathcal R^\star_{\mathrm{SLC}}(d,\alpha;\Lambda)
\gtrsim_\Lambda
\rho_{\alpha,d}(\Lambda).
\]

Conversely, for every fixed $P\in\mathcal P_{\mathrm{SLC}}(\alpha)$,
there exists a $P$-centered norm estimator
$\widehat\lambda_{\mathrm{norm},P}$ such that
\[
\sup_{\lambda\in\Lambda}
\mathbb E_{\lambda,P}
\left[
\bigl(
\widehat\lambda_{\mathrm{norm},P}(Y)-\lambda
\bigr)^2
\right]
\le
\min\left\{
\Delta_\Lambda^2,\,
\frac{\lambda_{\max}^4}{d}
\left(
\frac{4}{\alpha^2}
+
\frac{4}{\alpha\lambda_{\min}}
+
\frac{2}{\lambda_{\min}^2}
\right)
\right\}.
\]

The constrained fixed-$P$ MLE defined in Section~\ref{sec:statistical-formulation}
additionally satisfies
\[
\sup_{\lambda^\star\in\Lambda}
\mathbb E_{\lambda^\star,P}
\left[
\bigl(
\widehat\lambda_{\mathrm{MLE},P}(Y)-\lambda^\star
\bigr)^2
\right]
\lesssim_\Lambda
\min\left\{
\Delta_\Lambda^2,\,
\frac{(1+\alpha^{-1})^2\log(2d)}{d}
\right\}.
\]

Hence,
\[
\mathcal R^\star_{\mathrm{SLC}}(d,\alpha;\Lambda)
\asymp_\Lambda
\rho_{\alpha,d}(\Lambda).
\]
\end{theorem}

In the factor \(1+\alpha^{-1}\), the constant represents the Gaussian noise
scale and \(\alpha^{-1}\) measures the permitted signal spread. The factor
\(d^{-1}\) reflects averaging in high dimension. The risk saturates at the
squared interval width \(\Delta_\Lambda^2\).

\paragraph{Proof sketch.}
For the lower bound, it suffices to restrict to the Gaussian prior
$P_\alpha=\mathcal N(0,\alpha^{-1}I_d)$, for which the observation law remains Gaussian, and apply a two-point argument.

For the upper bound, writing $a=\lambda^{-1}$, consider
\[
\widetilde a_P(Y)
:=
\frac{\|Y-\mu_P\|^2-\mathbb E_P\|X-\mu_P\|^2}{d}.
\]
This statistic is unbiased for $a$. Strong log-concavity controls both
$\mathbb E_P\|X-\mu_P\|^2$ and
$\operatorname{Var}_P(\|X-\mu_P\|^2)$ through the Poincar\'e inequality, giving variance of order $(1+\alpha^{-1})^2/d$.
Projection onto the admissible variance interval followed by inversion gives the norm estimator.
For its analysis and the MLE argument, see Section~\ref{sec:upper-techniques} and Appendix~\ref{app:slc-proofs}.

\section{Techniques}
\label{sec:techniques}

We now give a technical overview of the main ideas behind the results in
Section~\ref{sec:results}. Complete proofs of the following arguments appear in
Appendices~\ref{app:ld-proofs} and \ref{app:slc-proofs}.

\subsection{Upper-bound techniques}
\label{sec:upper-techniques}

\paragraph{MLE under low intrinsic dimensionality.}
We localize the MLE through a quadratic
comparison of negative log-likelihoods. Recall that
$2\ell'(\lambda\mid Y)=-d/\lambda+D_\lambda(Y)$, where
$D_\lambda(Y)$ is the posterior quadratic error. The blind-denoising analysis gives high-probability control of
$D_\lambda(Y)$ around the true Gaussian noise energy
$d/\lambda^\star$ for each candidate $\lambda\in\Lambda$,
with fluctuations of order $\sqrt{d}+k$, up to logarithmic factors
and constants depending on $\Lambda$~\citep{kadkhodaie2026blinddenoising}.
Integrating the likelihood derivative therefore compares the observed
negative log-likelihood with the Gaussian reference function
\[
\Psi_{\lambda^\star}(\lambda)
=
\frac{d}{2}
\left(
\frac{\lambda}{\lambda^\star}
-1
-\log\frac{\lambda}{\lambda^\star}
\right).
\]
This function has curvature of order $d$, while the approximation error is
linear in $|\lambda-\lambda^\star|$. The excess negative log-likelihood
therefore has a lower bound of the form
$d(\lambda-\lambda^\star)^2-\mathrm{error}\cdot
|\lambda-\lambda^\star|$, and the defining inequality
$\ell(\widehat\lambda_{\mathrm{MLE},P}\mid Y)
\leq \ell(\lambda^\star\mid Y)$ localizes the MLE.

The posterior-energy bound is taken from
\citet{kadkhodaie2026blinddenoising}, where it is used to control posterior
draws. For the fixed-$P$ MLE, we combine this bound with the monotonicity of
$D_\lambda$. It is enough to control $D_\lambda(Y)$ at the endpoints of
$\Lambda$, which gives uniform control of the likelihood comparison over the
entire interval.

\paragraph{Norm estimator under strong log-concavity.}
For a norm statistic, the relevant quantity is the variance of the signal
energy. With $a=\lambda^{-1}$ and $Y=X+\sqrt{a}Z$, we have
\[
\operatorname{Var}(\|Y\|^2)
=
\operatorname{Var}_P(\|X\|^2)
+4a\,\mathbb{E}_P\|X\|^2
+2a^2d.
\]
Subtracting the mean signal energy under $P$ leaves the variance term
$\operatorname{Var}_P(\|X\|^2)$, which covering complexity $k$ does not control:
a distribution with small support complexity can assign mass to substantially
different radii.

Strong log-concavity provides precisely this missing variance control.
After centering at $\mu_P=\mathbb E_PX$, the Poincar\'e inequality
\citep{bakry2014analysis} gives
\[
\mathbb E_P\|X-\mu_P\|^2\leq \frac{d}{\alpha},
\qquad
\operatorname{Var}_P\bigl(\|X-\mu_P\|^2\bigr)
\leq \frac{4d}{\alpha^2}.
\]
The $P$-centered norm statistic therefore estimates the noise variance with
mean-squared error of order $(1+\alpha^{-1})^2/d$. Projection followed by
inversion gives the estimator in Theorem~\ref{thm:slc}.

\paragraph{MLE under strong log-concavity.}
The low-dimensional argument above relies on direct control of the posterior
energy $D_\lambda(Y)$. Under strong log-concavity, the same approach does not
preserve the desired dependence on $\alpha$, so we instead compare the noisy
marginals directly. 

Strong log-concavity is preserved under Gaussian convolution, with
inverse-curvature scale bounded by
\[
s_{\alpha,\Lambda}
:=
\alpha^{-1}+\lambda_{\min}^{-1}.
\]
The change in noise level produces a corresponding separation of the
covariances of $P_{\lambda^\star}$ and $P_\lambda$. Transportation and
covariance inequalities
\citep{saumard2014logconcavity,bakry2014analysis,gelbrich1990formula}
then give
\[
D_{1/2}(P_{\lambda^\star}\|P_\lambda)
\gtrsim_{\Lambda}
\frac{d(\lambda-\lambda^\star)^2}{s_{\alpha,\Lambda}^2}.
\]

This R\'enyi separation gives an exponential likelihood-ratio bound for each
fixed candidate $\lambda$. A one-dimensional net over $\Lambda$, together with
a derivative bound for $\lambda\mapsto\ell(\lambda\mid Y)$, extends the
comparison to the whole interval and localizes the MLE.

\subsection{Lower-bound techniques}
\label{sec:lower-techniques}

Beyond the universal parametric $d^{-1}$ lower bound, the
$k^2/d^2$ term requires a fixed signal law that conceals a change in
observation noise.

It is convenient to parameterize the noise by its variance
$a=\lambda^{-1}$. In the non-saturated regime, take a variance increment
$\delta\asymp_{\Lambda}k/d$ and consider the ideal signal mixture
\[
P^\star
=
\frac{1}{M}
\sum_{m=0}^{M-1}
\mathcal{N}(0,m\delta I_d),
\]
where $M$ is a sufficiently large absolute constant. Under observation-noise
variance $a_0$, the resulting marginal consists of Gaussian components with
variance levels
$a_0,a_0+\delta,\ldots,a_0+(M-1)\delta$. If the observation-noise variance is
increased to $a_1=a_0+\delta$, the same components appear at
$a_0+\delta,\ldots,a_0+M\delta$. Hence the $M-1$ interior components coincide
exactly, and
\[
Q_0^\star-Q_1^\star
=
\frac{1}{M}
\left\{
\mathcal{N}(0,a_0I_d)
-
\mathcal{N}(0,(a_0+M\delta)I_d)
\right\}.
\]
Therefore $\|Q_0^\star-Q_1^\star\|_{\mathrm{TV}}\leq 1/M$.
Figure~\ref{fig:ld-lower-construction} illustrates the overlapping components.

\begin{figure}[t]
    \centering
    \includegraphics[width=0.75\linewidth]
    {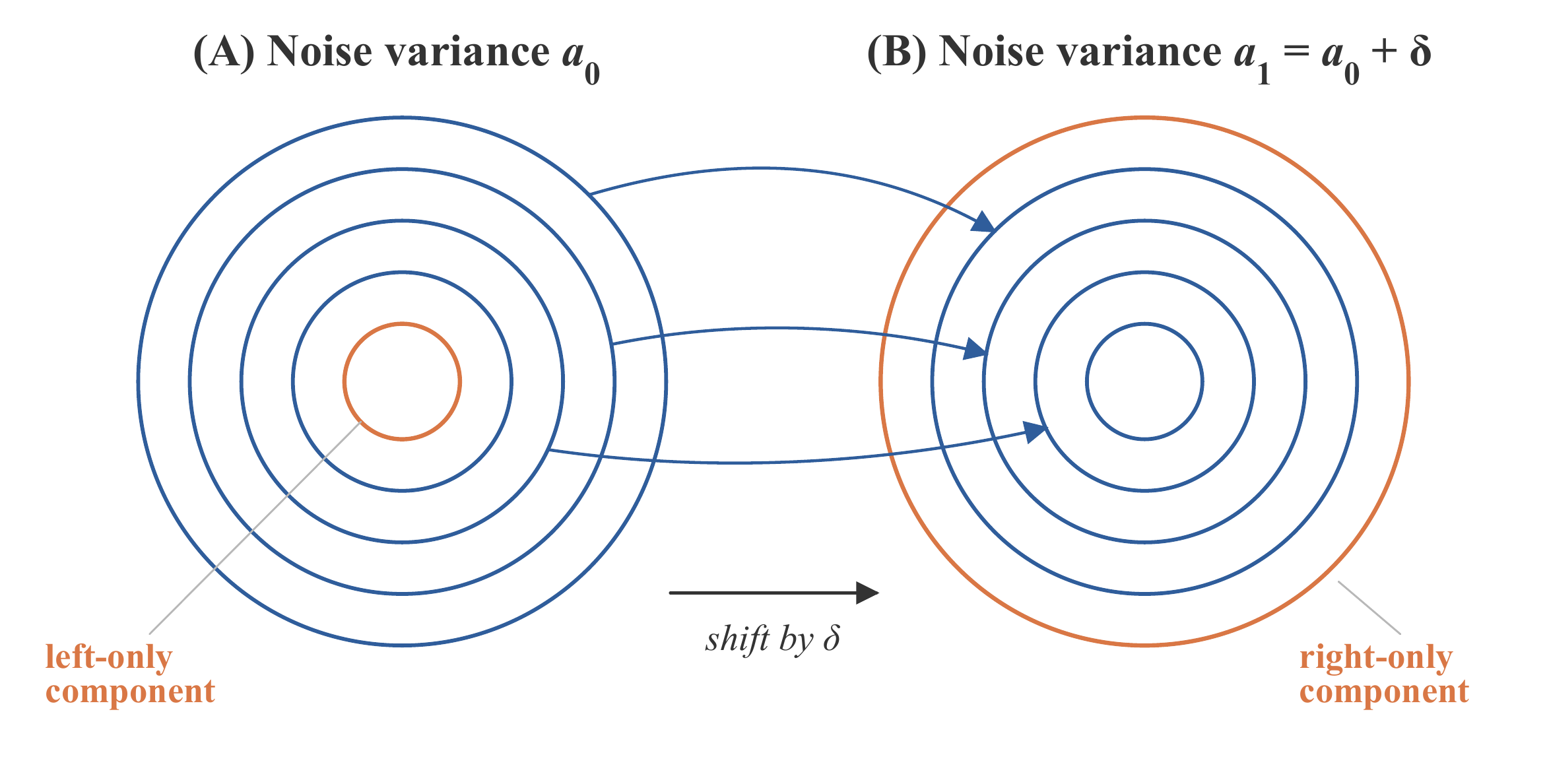}
    \caption{\textbf{Intrinsic-dimensional lower-bound construction.}
    Increasing the observation-noise variance from $a_0$ to $a_0+\delta$
    shifts each Gaussian component by one variance increment. The $M-1$
    interior components align across the two observation laws, leaving only the
    two endpoint components unmatched, each with weight $1/M$.}
    \label{fig:ld-lower-construction}
\end{figure}

This mixture construction is related to lower bounds for noise-level
estimation in semiparametric convolution models,
where latent Gaussian variability can mask changes in the noise variance
\citep{matias2002semiparametric,butucea2005minimax}. Those constructions,
however, allow the signal law to vary between the hypotheses, whereas here the
same fixed signal law must conceal both noise levels.

The ideal mixture has unbounded support and therefore does not satisfy the
covering-number condition. We replace each Gaussian component by a finite
empirical point cloud whose Gaussian smoothing is close to the corresponding
population component in $\chi^2$ divergence. A point-cloud size exponential
in $k$ gives the required approximation while keeping the total number of
support points below $e^k$, and Gaussian concentration keeps the support
radius at order $\sqrt{k}$. This approximation uses the same Gaussian-smoothed empirical-measure
principle studied by \citet{goldfeld2020convergence}, with the approximation
measured here in $\chi^2$ divergence. The resulting deterministic prior belongs to
$\mathcal{P}_{\mathrm{LD}}(k,R)$ and preserves a constant overlap in total
variation between the two observation laws. Since
$\delta\asymp_{\Lambda}k/d$ corresponds to a precision separation of the same
order on the compact interval $\Lambda$, the construction yields the
$k^2/d^2$ term in Theorem~\ref{thm:ld-main} through the standard two-point method~\citep{tsybakov2009introduction}.

For strong log-concavity, the class contains the Gaussian
law \(P_\alpha=\mathcal{N}(0,\alpha^{-1}I_d)\), under which
\(Y\sim\mathcal{N}(0,(\alpha^{-1}+\lambda^{-1})I_d)\). A standard two-point
comparison within this submodel gives the scale
\((1+\alpha^{-1})^2/d\)~\citep{tsybakov2009introduction}.

\section{Conclusion}
\label{sec:conclusion}

In this work, we characterized the minimax risk of estimating a Gaussian noise level from one
high-dimensional observation with a known, structurally constrained signal law.
Under low intrinsic dimensionality, in the regime of
Theorem~\ref{thm:ld-main}, the minimax risk is
$\widetilde{\Theta}_{\Lambda}(d^{-1}+k^2d^{-2})$ up to saturation by
$\Delta_\Lambda^2$, and the fixed-$P$ MLE attains this rate up to logarithmic
factors. Under strong log-concavity, the minimax risk is
$\Theta_{\Lambda}((1+\alpha^{-1})^2/d)$ up to the same saturation, and a
$P$-centered norm estimator is minimax optimal. Thus, low-dimensional support
is sufficient but not necessary for accurate estimation of the noise level. More broadly, these results show that the difficulty of blind noise estimation
is governed by the structure of the signal law rather than by ambient dimension
alone, and that different structural assumptions naturally lead to different
estimation procedures.

The analysis assumes an exactly known signal law and isotropic Gaussian noise
with a single unknown scalar level. How the rates change when the signal law
is learned or misspecified, as in practical diffusion models, remains open.
Other directions include anisotropic or non-Gaussian corruption, as well as
signal classes beyond low intrinsic dimensionality and strong log-concavity. A broader open question is to characterize which properties of the signal law
make blind noise estimation possible. The separation-based argument for the
strongly log-concave case suggests one possible route: establish direct
quantitative separation between the corresponding noisy marginals.

\subsection*{Acknowledgments}

I am deeply grateful to Sinho Chewi and Aram-Alexandre Pooladian for their guidance throughout this project. Their questions, suggestions, and careful feedback helped shape the direction of the work and substantially improved the manuscript. I am especially thankful for the time and care they devoted to this work.

\subsection*{AI use statement}

Generative AI tools were used in this work to assist with research ideation, literature retrieval and discovery, translation, drafting and polishing portions of the manuscript, and checking and developing mathematical arguments and proofs. AI-assisted suggestions were independently reviewed against the relevant literature and mathematical derivations. In particular, all mathematical claims and proofs were independently checked, and all cited references were independently verified. Full responsibility is taken for the final content of the paper, including all statements, proofs, and citations.

\subsection*{Reproducibility statement}

This work is theoretical and does not rely on empirical datasets or numerical
experiments. The statistical formulation and assumptions are given in
Section~\ref{sec:background}, and the main results are stated in
Section~\ref{sec:results}. Section~\ref{sec:techniques} provides an overview of
the proof techniques. Complete proofs and supporting lemmas are provided in
Appendices~\ref{app:preliminaries}, \ref{app:ld-proofs}, and
\ref{app:slc-proofs}.

\bibliography{iclr2027_conference}

@book{tsybakov2009introduction,
  title     = {Introduction to Nonparametric Estimation},
  author    = {Tsybakov, Alexandre B.},
  year      = {2009},
  publisher = {Springer},
  series    = {Springer Series in Statistics},
}

@book{bakry2014analysis,
  title     = {Analysis and Geometry of Markov Diffusion Operators},
  author    = {Bakry, Dominique and Gentil, Ivan and Ledoux, Michel},
  year      = {2014},
  publisher = {Springer},
  series    = {Grundlehren der mathematischen Wissenschaften},
  volume    = {348},
}

@article{saumard2014logconcavity,
  title     = {Log-Concavity and Strong Log-Concavity: A Review},
  author    = {Saumard, Adrien and Wellner, Jon A.},
  journal   = {Statistics Surveys},
  volume    = {8},
  pages     = {45--114},
  year      = {2014},
}

@article{gelbrich1990formula,
  title = {On a Formula for the {$L^2$} {Wasserstein} Metric between Measures on {Euclidean} and {Hilbert} Spaces},
  author    = {Gelbrich, Matthias},
  journal   = {Mathematische Nachrichten},
  volume    = {147},
  number    = {1},
  pages     = {185--203},
  year      = {1990},
}

@inproceedings{sohldickstein2015deep,
  title     = {Deep Unsupervised Learning using Nonequilibrium Thermodynamics},
  author    = {Sohl-Dickstein, Jascha and Weiss, Eric and Maheswaranathan, Niru and Ganguli, Surya},
  booktitle = {Proceedings of the 32nd International Conference on Machine Learning},
  series    = {Proceedings of Machine Learning Research},
  volume    = {37},
  pages     = {2256--2265},
  publisher = {PMLR},
  year      = {2015}
}

@inproceedings{song2021scorebased,
  title     = {Score-Based Generative Modeling through Stochastic Differential Equations},
  author    = {Song, Yang and Sohl-Dickstein, Jascha and Kingma, Diederik P. and Kumar, Abhishek and Ermon, Stefano and Poole, Ben},
  booktitle = {International Conference on Learning Representations},
  year      = {2021}
}

@inproceedings{gnanasambandam2020one,
  title     = {One Size Fits All: Can We Train One Denoiser for All Noise Levels?},
  author    = {Gnanasambandam, Abhiram and Chan, Stanley H.},
  booktitle = {Proceedings of the 37th International Conference on Machine Learning},
  series    = {Proceedings of Machine Learning Research},
  volume    = {119},
  pages     = {3576--3586},
  publisher = {PMLR},
  year      = {2020}
}

@article{miyasawa1961empirical,
  title   = {An Empirical Bayes Estimator of the Mean of a Normal Population},
  author  = {Miyasawa, Koichi},
  journal = {Bulletin of the International Statistical Institute},
  volume  = {38},
  pages   = {181--188},
  year    = {1961}
}

@article{efron2011tweedie,
  title     = {Tweedie's Formula and Selection Bias},
  author    = {Efron, Bradley},
  journal   = {Journal of the American Statistical Association},
  volume    = {106},
  number    = {496},
  pages     = {1602--1614},
  publisher = {Taylor \& Francis},
  year      = {2011}
}

@article{kadkhodaie2026blinddenoising,
  title   = {Blind Denoising Diffusion Models and the Blessings of Dimensionality},
  author  = {Kadkhodaie, Zahra and Pooladian, Aram-Alexandre and Chewi, Sinho and Simoncelli, Eero P.},
  journal = {arXiv preprint arXiv:2602.09639},
  year    = {2026}
}

@article{matias2002semiparametric,
  title   = {Semiparametric Deconvolution with Unknown Noise Variance},
  author  = {Matias, Catherine},
  journal = {ESAIM: Probability and Statistics},
  volume  = {6},
  pages   = {271--292},
  year    = {2002}
}

@article{butucea2005minimax,
  title   = {Minimax Estimation of the Noise Level and of the Deconvolution Density in a Semiparametric Convolution Model},
  author  = {Butucea, Cristina and Matias, Catherine},
  journal = {Bernoulli},
  volume  = {11},
  number  = {2},
  pages   = {309--340},
  year    = {2005}
}

@article{goldfeld2020convergence,
  title   = {Convergence of Smoothed Empirical Measures with Applications to Entropy Estimation},
  author  = {Goldfeld, Ziv and Greenewald, Kristjan and Niles-Weed, Jonathan and Polyanskiy, Yury},
  journal = {IEEE Transactions on Information Theory},
  volume  = {66},
  number  = {7},
  pages   = {4368--4391},
  year    = {2020}
}

@inproceedings{ho2020denoising,
  title     = {Denoising Diffusion Probabilistic Models},
  author    = {Ho, Jonathan and Jain, Ajay and Abbeel, Pieter},
  booktitle = {Advances in Neural Information Processing Systems},
  volume    = {33},
  pages     = {6840--6851},
  year      = {2020}
}

@article{liu2013single,
  title     = {Single-Image Noise Level Estimation for Blind Denoising},
  author    = {Liu, Xinhao and Tanaka, Masayuki and Okutomi, Masatoshi},
  journal   = {IEEE Transactions on Image Processing},
  volume    = {22},
  number    = {12},
  pages     = {5226--5237},
  year      = {2013}
}

@inproceedings{kadkhodaie2021stochastic,
  title     = {Stochastic Solutions for Linear Inverse Problems using the Prior Implicit in a Denoiser},
  author    = {Kadkhodaie, Zahra and Simoncelli, Eero P.},
  booktitle = {Advances in Neural Information Processing Systems},
  volume    = {34},
  pages     = {13242--13254},
  year      = {2021}
}

@inproceedings{sun2025noiseconditioning,
  title     = {Is Noise Conditioning Necessary for Denoising Generative Models?},
  author    = {Sun, Qiao and Jiang, Zhicheng and Zhao, Hanhong and He, Kaiming},
  booktitle = {Proceedings of the 42nd International Conference on Machine Learning},
  series    = {Proceedings of Machine Learning Research},
  volume    = {267},
  pages     = {57469--57502},
  publisher = {PMLR},
  year      = {2025}
}

@inproceedings{heurteldepeiges2024listening,
  title     = {Listening to the Noise: Blind Denoising with {G}ibbs Diffusion},
  author    = {Heurtel-Depeiges, David and Margossian, Charles and Ohana, Ruben and R{\'e}galdo-Saint Blancard, Bruno},
  booktitle = {Proceedings of the 41st International Conference on Machine Learning},
  series    = {Proceedings of Machine Learning Research},
  volume    = {235},
  pages     = {18284--18304},
  publisher = {PMLR},
  year      = {2024}
}

@inproceedings{pope2021intrinsic,
  title     = {The Intrinsic Dimension of Images and Its Impact on Learning},
  author    = {Pope, Phillip and Zhu, Chen and Abdelkader, Ahmed and Goldblum, Micah and Goldstein, Tom},
  booktitle = {International Conference on Learning Representations},
  year      = {2021}
}

@article{huang2024denoising,
  title   = {Denoising Diffusion Probabilistic Models Are Optimally Adaptive to Unknown Low Dimensionality},
  author  = {Huang, Zhihan and Wei, Yuting and Chen, Yuxin},
  journal = {Mathematics of Operations Research},
  year    = {2026}
}

@inproceedings{li2024adapting,
  title     = {Adapting to Unknown Low-Dimensional Structures in Score-Based Diffusion Models},
  author    = {Li, Gen and Yan, Yuling},
  booktitle = {Advances in Neural Information Processing Systems},
  volume    = {37},
  pages     = {126297--126331},
  year      = {2024}
}

@inproceedings{potaptchik2025linear,
  title     = {Linear Convergence of Diffusion Models under the Manifold Hypothesis},
  author    = {Potaptchik, Peter and Azangulov, Iskander and Deligiannidis, George},
  booktitle = {Proceedings of the Thirty-Eighth Conference on Learning Theory},
  series    = {Proceedings of Machine Learning Research},
  volume    = {291},
  pages     = {4668--4685},
  publisher = {PMLR},
  year      = {2025}
}

@inproceedings{liang2025lowdimensional,
  title     = {Low-Dimensional Adaptation of Diffusion Models: Convergence in Total Variation (extended abstract)},
  author    = {Liang, Jiadong and Huang, Zhihan and Chen, Yuxin},
  booktitle = {Proceedings of the Thirty-Eighth Conference on Learning Theory},
  series    = {Proceedings of Machine Learning Research},
  volume    = {291},
  pages     = {3723--3729},
  publisher = {PMLR},
  year      = {2025}
}

@inproceedings{chen2023sampling,
  title     = {Sampling Is as Easy as Learning the Score: Theory for Diffusion Models with Minimal Data Assumptions},
  author    = {Chen, Sitan and Chewi, Sinho and Li, Jerry and Li, Yuanzhi and Salim, Adil and Zhang, Anru R.},
  booktitle = {International Conference on Learning Representations},
  year      = {2023}
}

@inproceedings{chen2023improved,
  title     = {Improved Analysis of Score-Based Generative Modeling: User-Friendly Bounds under Minimal Smoothness Assumptions},
  author    = {Chen, Hongrui and Lee, Holden and Lu, Jianfeng},
  booktitle = {Proceedings of the 40th International Conference on Machine Learning},
  series    = {Proceedings of Machine Learning Research},
  volume    = {202},
  pages     = {4735--4763},
  publisher = {PMLR},
  year      = {2023}
}

@inproceedings{lee2023convergence,
  title     = {Convergence of Score-Based Generative Modeling for General Data Distributions},
  author    = {Lee, Holden and Lu, Jianfeng and Tan, Yixin},
  booktitle = {Proceedings of the 34th International Conference on Algorithmic Learning Theory},
  series    = {Proceedings of Machine Learning Research},
  volume    = {201},
  pages     = {946--985},
  publisher = {PMLR},
  year      = {2023}
}

@inproceedings{wibisono2024optimal,
  title     = {Optimal Score Estimation via Empirical Bayes Smoothing},
  author    = {Wibisono, Andre and Wu, Yihong and Yang, Kaylee Yingxi},
  booktitle = {Proceedings of the Thirty-Seventh Conference on Learning Theory},
  series    = {Proceedings of Machine Learning Research},
  volume    = {247},
  pages     = {4958--4991},
  publisher = {PMLR},
  year      = {2024}
}

@inproceedings{lipman2023flow,
  title     = {Flow Matching for Generative Modeling},
  author    = {Lipman, Yaron and Chen, Ricky T. Q. and Ben-Hamu, Heli and Nickel, Maximilian and Le, Matt},
  booktitle = {International Conference on Learning Representations},
  year      = {2023}
}

@article{lee2026beckmann,
  title   = {Beckmann Transport Models: From Autonomous Flows to One-Step Maps},
  author  = {Lee, Cheuk-Kit and Coeurdoux, Florentin and Chen, Yuyuan and Tang, Sophia and Potaptchik, Peter and Du, Yilun and Albergo, Michael Samuel and Vanden-Eijnden, Eric},
  journal = {arXiv preprint arXiv:2608.01692},
  year    = {2026}
}

@inproceedings{song2019generative,
  title     = {Generative Modeling by Estimating Gradients of the Data Distribution},
  author    = {Song, Yang and Ermon, Stefano},
  booktitle = {Advances in Neural Information Processing Systems},
  volume    = {32},
  year      = {2019},
}

@inproceedings{albergo2023building,
  title     = {Building Normalizing Flows with Stochastic Interpolants},
  author    = {Albergo, Michael S. and Vanden-Eijnden, Eric},
  booktitle = {International Conference on Learning Representations},
  year      = {2023},
}

@inproceedings{liu2023flow,
  title     = {Flow Straight and Fast: Learning to Generate and Transfer Data with Rectified Flow},
  author    = {Liu, Xingchao and Gong, Chengyue and Liu, Qiang},
  booktitle = {International Conference on Learning Representations},
  year      = {2023},
}

@article{zhang2017beyond,
  title = {Beyond a {Gaussian} Denoiser: Residual Learning of Deep {CNN} for Image Denoising},
  author  = {Zhang, Kai and Zuo, Wangmeng and Chen, Yunjin and Meng, Deyu and Zhang, Lei},
  journal = {IEEE Transactions on Image Processing},
  volume  = {26},
  number  = {7},
  pages   = {3142--3155},
  year    = {2017},
}

@article{zhang2018ffdnet,
  title = {{FFDNet}: Toward a Fast and Flexible Solution for {CNN}-Based Image Denoising},
  author  = {Zhang, Kai and Zuo, Wangmeng and Zhang, Lei},
  journal = {IEEE Transactions on Image Processing},
  volume  = {27},
  number  = {9},
  pages   = {4608--4622},
  year    = {2018},
}

@inproceedings{dhariwal2021diffusion,
  title = {Diffusion Models Beat {GANs} on Image Synthesis},
  author    = {Dhariwal, Prafulla and Nichol, Alexander},
  booktitle = {Advances in Neural Information Processing Systems},
  volume    = {34},
  pages     = {8780--8794},
  year      = {2021},
}

@inproceedings{chen2015efficient,
  title     = {An Efficient Statistical Method for Image Noise Level Estimation},
  author    = {Chen, Guangyong and Zhu, Fengyuan and Heng, Pheng-Ann},
  booktitle = {Proceedings of the IEEE International Conference on Computer Vision},
  pages     = {477--485},
  year      = {2015},
}
\bibliographystyle{iclr2027_conference}

\appendix
\section{Preliminaries}
\label{app:preliminaries}

Throughout the appendix, let
\[
    \Lambda=[\lambda_{\min},\lambda_{\max}],
    \qquad
    \Delta_{\Lambda}:=\lambda_{\max}-\lambda_{\min},
\]
and write \(A\lesssim_{\Lambda} B\) when \(A\leq C_{\Lambda}B\) for a
constant depending only on the fixed interval \(\Lambda\). This appendix gives the likelihood identities, testing inequalities, and
moment bounds used to prove Theorems~\ref{thm:ld-main} and
\ref{thm:slc}.

\subsection{Observation law and likelihood identities}
\label{app:likelihood-identities}

Let \(P\) be a probability measure on \(\mathbb{R}^{d}\), and consider
\[
    Y=X+\lambda^{-1/2}Z,
    \qquad
    X\sim P,
    \qquad
    Z\sim\mathcal{N}(0,I_d),
\]
where \(X\) and \(Z\) are independent. The observation law is
\[
    P_{\lambda}:=P*\mathcal{N}(0,\lambda^{-1}I_d),
\]
with density
\[
    p_{\lambda}(y)
    =
    \left(\frac{\lambda}{2\pi}\right)^{d/2}
    \int
    \exp\left(-\frac{\lambda}{2}\|x-y\|^{2}\right)
    P(dx).
\]
The posterior distribution of \(X\) given \(Y=y\) is
\[
    q_{\lambda}(dx\mid y)
    =
    \frac{
        \exp\left(-\frac{\lambda}{2}\|x-y\|^{2}\right)P(dx)
    }{
        \int
        \exp\left(-\frac{\lambda}{2}\|u-y\|^{2}\right)P(du)
    }.
\]
Define the negative log-likelihood and posterior quadratic error by
\[
    \ell(\lambda\mid y):=-\log p_{\lambda}(y),
    \qquad
    D_{\lambda}(y)
    :=
    \mathbb{E}_{q_{\lambda}(\cdot\mid y)}
    \|X-y\|^{2}.
\]

\begin{lemma}[Score and monotonicity identities]
\label{lem:score-monotonicity}
For every \(y\in\mathbb{R}^{d}\) and \(\lambda>0\),
\[
    2\ell'(\lambda\mid y)
    =
    -\frac{d}{\lambda}+D_{\lambda}(y),
\]
and
\[
    \partial_{\lambda}D_{\lambda}(y)
    =
    -\frac{1}{2}
    \operatorname{Var}_{q_{\lambda}(\cdot\mid y)}
    \bigl(\|X-y\|^{2}\bigr)
    \leq 0.
\]
In particular, \(\lambda\mapsto D_{\lambda}(y)\) is nonincreasing.
\end{lemma}

\begin{proof}
Differentiating the marginal log-density gives
\[
    \partial_{\lambda}\log p_{\lambda}(y)
    =
    \frac{d}{2\lambda}
    -
    \frac{1}{2}
    \mathbb{E}_{q_{\lambda}(\cdot\mid y)}
    \|X-y\|^{2},
\]
which proves the score identity. More generally, for any integrable
function \(h\),
\[
    \partial_{\lambda}
    \mathbb{E}_{q_{\lambda}(\cdot\mid y)}h(X)
    =
    -\frac{1}{2}
    \operatorname{Cov}_{q_{\lambda}(\cdot\mid y)}
    \bigl(h(X),\|X-y\|^{2}\bigr).
\]
Taking \(h(X)=\|X-y\|^{2}\) gives the second identity.
\end{proof}

\subsection{Divergences and two-point lower bounds}
\label{app:testing-tools}

For probability measures \(Q\) and \(Q'\), we use
\[
    \|Q-Q'\|_{\mathrm{TV}}
    :=
    \sup_{A}|Q(A)-Q'(A)|,
    \qquad
    \operatorname{KL}(Q\|Q')
    :=
    \int\log\left(\frac{dQ}{dQ'}\right)dQ,
\]
and
\[
    \chi^{2}(Q\|Q')
    :=
    \int
    \left(\frac{dQ}{dQ'}-1\right)^{2}dQ'.
\]
The order-\(1/2\) Rényi divergence is
\[
    D_{1/2}(Q\|Q')
    :=
    -2\log
    \int
    \sqrt{\frac{dQ}{d\nu}\frac{dQ'}{d\nu}}\,d\nu,
\]
where \(\nu\) is any common dominating measure.

We repeatedly use the standard comparisons
\[
    \|Q-Q'\|_{\mathrm{TV}}
    \leq
    \sqrt{\frac{1}{2}\operatorname{KL}(Q\|Q')},
    \qquad
    \|Q-Q'\|_{\mathrm{TV}}
    \leq
    \frac{1}{2}\sqrt{\chi^{2}(Q\|Q')},
\]
as well as the data-processing inequality for relative entropy.

\begin{lemma}[Two-point lower bound]
\label{lem:two-point}
Let \(Q_0,Q_1\) be two observation laws associated with parameters
\(\theta_0,\theta_1\in\mathbb{R}\). Then
\[
    \inf_{\widehat{\theta}}
    \max_{i\in\{0,1\}}
    \mathbb{E}_{Q_i}
    \bigl(\widehat{\theta}-\theta_i\bigr)^2
    \geq
    \frac{|\theta_1-\theta_0|^{2}}{8}
    \bigl(1-\|Q_0-Q_1\|_{\mathrm{TV}}\bigr).
\]
\end{lemma}

\begin{proof}
This is the standard two-point reduction; see
\citet[Chapter~2]{tsybakov2009introduction}.
\end{proof}

For Gaussian laws with a common mean, if \(v_0,v_1>0\), then
\[
    \operatorname{KL}
    \left(
        \mathcal{N}(m,v_0I_d)
        \,\middle\|\,
        \mathcal{N}(m,v_1I_d)
    \right)
    =
    \frac{d}{2}
    \left(
        \frac{v_0}{v_1}
        -1
        -\log\frac{v_0}{v_1}
    \right).
\]
The common mean does not affect the divergence.

\subsection{Radial moments and strong log-concavity}
\label{app:radial-preliminaries}

The variance of the norm statistic depends on the signal's radial moments,
as the following identity shows.

\begin{lemma}[Radial variance identity]
\label{lem:radial-variance}
Suppose that \(X\) has finite fourth moment, let \(m\in\mathbb{R}^d\)
be deterministic, and write \(a=\lambda^{-1}\). Then
\[
    \mathbb{E}_{\lambda,P}\|Y-m\|^{2}
    =
    \mathbb{E}_P\|X-m\|^{2}+ad,
\]
and
\[
    \operatorname{Var}_{\lambda,P}
    \bigl(\|Y-m\|^{2}\bigr)
    =
    \operatorname{Var}_P
    \bigl(\|X-m\|^{2}\bigr)
    +
    4a\,\mathbb{E}_P\|X-m\|^{2}
    +
    2a^{2}d.
\]
Taking \(m=0\) gives the identity used in the discussion of low intrinsic
dimensionality; taking \(m=\mu_P\) gives its centered form for the
strongly log-concave case.
\end{lemma}

\begin{proof}
Set \(U:=X-m\). Since
\[
    Y-m=U+\sqrt{a}\,Z,
\]
we have
\[
    \|Y-m\|^{2}
    =
    \|U\|^{2}
    +
    2\sqrt{a}\langle U,Z\rangle
    +
    a\|Z\|^{2}.
\]
The three terms on the right are pairwise uncorrelated. Moreover,
\[
    \mathbb{E}\langle U,Z\rangle^{2}
    =
    \mathbb{E}_P\|U\|^{2},
    \qquad
    \operatorname{Var}(\|Z\|^{2})=2d.
\]
Taking expectations and variances proves the identities.
\end{proof}

Strong log-concavity gives the following bounds on these moments.

\begin{lemma}[Consequences of strong log-concavity]
\label{lem:slc-moments}
Let
\[
    P(dx)=Z_V^{-1}e^{-V(x)}dx,
    \qquad
    \nabla^{2}V(x)\succeq \alpha I_d.
\]
Then
\[
    \operatorname{Cov}_P(X)\preceq \alpha^{-1}I_d,
    \qquad
    \mathbb{E}_P\|X-\mu_P\|^{2}\leq \frac{d}{\alpha},
\]
and
\[
    \operatorname{Var}_P
    \bigl(\|X-\mu_P\|^{2}\bigr)
    \leq
    \frac{4d}{\alpha^{2}}.
\]
\end{lemma}

\begin{proof}
Strong log-concavity implies the Poincaré inequality
\[
    \operatorname{Var}_P(f)
    \leq
    \frac{1}{\alpha}
    \mathbb{E}_P\|\nabla f(X)\|^{2}
\]
for every sufficiently regular \(f\); see
\citet{bakry2014analysis}. Applying it to
\(f_v(x)=\langle v,x\rangle\) yields
\[
    v^{\top}\operatorname{Cov}_P(X)v
    \leq
    \frac{\|v\|^{2}}{\alpha},
\]
and hence
\[
    \mathbb{E}_P\|X-\mu_P\|^{2}
    =
    \operatorname{tr}\operatorname{Cov}_P(X)
    \leq
    \frac{d}{\alpha}.
\]
Finally, applying the Poincaré inequality to
\(f(x)=\|x-\mu_P\|^{2}\) gives
\[
    \operatorname{Var}_P
    \bigl(\|X-\mu_P\|^{2}\bigr)
    \leq
    \frac{4}{\alpha}
    \mathbb{E}_P\|X-\mu_P\|^{2}
    \leq
    \frac{4d}{\alpha^{2}}.
\]
\end{proof}

\section{Proofs for Low Intrinsic Dimensionality}
\label{app:ld-proofs}

Throughout this section, let
\[
    \Lambda=[\lambda_{\min},\lambda_{\max}],
    \qquad
    \Delta_{\Lambda}:=\lambda_{\max}-\lambda_{\min},
    \qquad
    r_0:=
    \frac{\sqrt{\lambda_{\min}}}
         {\lambda_{\max}\sqrt d},
\]
and write
\[
    P_{\lambda}
    :=
    P * \mathcal{N}(0,\lambda^{-1}I_d)
\]
for the observation law under precision $\lambda$.
To prove Theorem~\ref{thm:ld-main}, we establish the universal parametric
lower bound, construct the additional lower bound from covering complexity,
and bound the MLE risk. The final subsection gives the norm-statistic
calculation used in Section~\ref{sec:upper-techniques}.

\subsection{Universal parametric lower bound}
\label{app:ld-parametric-lb}

The parametric lower bound holds for every fixed signal law $P$.

\begin{proposition}[Universal parametric lower bound]
\label{prop:ld-universal-lb}
For every probability measure $P$ on $\mathbb{R}^d$,
\[
    \mathcal{R}(P;\Lambda)
    \gtrsim
    \min\left\{
        \Delta_{\Lambda}^2,\,
        \frac{\lambda_{\min}^2}{d}
    \right\}.
\]
Consequently,
\[
    \mathcal{R}^{\star}_{\mathrm{LD}}(d,k,R;\Lambda)
    \gtrsim
    \min\left\{
        \Delta_{\Lambda}^2,\,
        \frac{\lambda_{\min}^2}{d}
    \right\}.
\]
\end{proposition}

\begin{proof}
Fix a signal law $P$ and consider two nearby precision parameters. Let
\[
    \lambda_0:=\lambda_{\min},
    \qquad
    \lambda_1:=\lambda_0+2\delta,
\]
where we choose $\delta>0$ below to keep the two observation laws close.

Let $\mu_i:=P_{\lambda_i}$, $i\in\{0,1\}$. Augment the observation $Y$ with the latent signal $X$. Under hypothesis $i$, the
joint law of $(X,Y)$ is
\[
    \widetilde{\mu}_i(dx,dy)
    =
    P(dx)\,
    \mathcal{N}(x,\lambda_i^{-1}I_d)(dy).
\]
Since $\mu_i$ is the $Y$-marginal of $\widetilde{\mu}_i$, the
data-processing inequality gives
\[
    \mathrm{KL}(\mu_0\|\mu_1)
    \leq
    \mathrm{KL}(\widetilde{\mu}_0\|\widetilde{\mu}_1).
\]
The distribution of $X$ is the same under both hypotheses, so the
chain rule for relative entropy yields
\begin{align*}
    \mathrm{KL}(\widetilde{\mu}_0\|\widetilde{\mu}_1)
    &=
    \mathbb{E}_{X\sim P}
    \mathrm{KL}\left(
        \mathcal{N}(X,\lambda_0^{-1}I_d)
        \,\big\|\,
        \mathcal{N}(X,\lambda_1^{-1}I_d)
    \right) \\
    &=
    \mathrm{KL}\left(
        \mathcal{N}(0,\lambda_0^{-1}I_d)
        \,\big\|\,
        \mathcal{N}(0,\lambda_1^{-1}I_d)
    \right) \\
    &=
    \frac{d}{2}
    \left(
        \frac{\lambda_1}{\lambda_0}
        -1
        -\log\frac{\lambda_1}{\lambda_0}
    \right).
\end{align*}
Thus the divergence does not depend on the common mean.

Choose the separation by setting
\[
    \delta
    :=
    c\min\left\{
        \Delta_{\Lambda},\,
        \frac{\lambda_{\min}}{\sqrt d}
    \right\},
\]
where $c>0$ is a sufficiently small absolute constant. Then
$\lambda_1\in\Lambda$, and
$2\delta/\lambda_0\leq 2c$. Using
$u-\log(1+u)\leq Cu^2$ for sufficiently small $u\geq0$, we obtain
\[
    \mathrm{KL}(\mu_0\|\mu_1)
    \leq
    C d\frac{\delta^2}{\lambda_0^2}
    \leq Cc^2.
\]
For sufficiently small $c$, the KL divergence is bounded by a fixed
numerical constant. Pinsker's inequality then gives
$\|\mu_0-\mu_1\|_{\mathrm{TV}}\leq c_0$ for some absolute
$c_0<1$.

The standard two-point reduction for squared-error estimation
therefore implies
\[
    \inf_{\widehat{\lambda}_P}
    \sup_{i\in\{0,1\}}
    \mathbb{E}_{\mu_i}
    \left[
        (\widehat{\lambda}_P-\lambda_i)^2
    \right]
    \gtrsim
    |\lambda_1-\lambda_0|^2
    \bigl(
        1-\|\mu_0-\mu_1\|_{\mathrm{TV}}
    \bigr),
\]
see, e.g., \citet[Chapter~2]{tsybakov2009introduction}. Since
$|\lambda_1-\lambda_0|=2\delta$, this yields
\[
    \mathcal{R}(P;\Lambda)
    \gtrsim
    \delta^2
    \asymp
    \min\left\{
        \Delta_{\Lambda}^2,\,
        \frac{\lambda_{\min}^2}{d}
    \right\}.
\]
Since $P$ was arbitrary, we can take the supremum over
$P\in\mathcal{P}_{\mathrm{LD}}(k,R)$.
\end{proof}

\subsection{Lower bound from covering complexity}
\label{app:ld-intrinsic-lb}

For the second lower bound, we construct one signal law with low covering
complexity for which the observation laws remain close in total variation
at the required precision separation.

\begin{lemma}[Intrinsic-dimensional indistinguishable pair]
\label{lem:ld-indistinguishable-pair}
There exists a sufficiently large absolute constant $k_0$ such that
the following holds. Assume $k_0\leq k\lesssim_{\Lambda} d$ and
$R^2\gtrsim_{\Lambda} k$. Then there exist
$P\in\mathcal{P}_{\mathrm{LD}}(k,R)$ and
$\lambda_0,\lambda_1\in\Lambda$ such that
\[
    |\lambda_1-\lambda_0|
    \gtrsim
    \min\left\{
        \Delta_{\Lambda},\,
        \lambda_{\min}\frac{k}{d}
    \right\},
\]
while
\[
    \left\|
        P_{\lambda_0}-P_{\lambda_1}
    \right\|_{\mathrm{TV}}
    \leq c_0
\]
for some absolute constant $c_0<1$.
\end{lemma}

The construction has an ideal Gaussian mixture and a finite-support
approximation.

\paragraph{Ideal Gaussian-mixture construction.}
It is convenient to work with the noise variance
$a=\lambda^{-1}$. Recall that the admissible variance interval is
\[
    \mathcal{A}_{\Lambda}
    :=
    [\lambda_{\max}^{-1},\lambda_{\min}^{-1}].
\]
Set
\[
    a_0:=\lambda_{\max}^{-1},
    \qquad
    \tau
    :=
    \min\left\{
        \frac{k}{d},\,
        \frac{\Delta_{\Lambda}}{\lambda_{\min}},\,
        1
    \right\},
\]
and let
\[
    \delta:=c a_0\tau,
    \qquad
    a_1:=a_0+\delta,
    \qquad
    \lambda_i:=a_i^{-1},
    \quad i\in\{0,1\},
\]
where $c>0$ is a sufficiently small absolute constant.

Since
\[
    \lambda_{\min}^{-1}-\lambda_{\max}^{-1}
    =
    \frac{\Delta_{\Lambda}}
         {\lambda_{\min}\lambda_{\max}}
    =
    a_0\frac{\Delta_{\Lambda}}{\lambda_{\min}},
\]
the definition of $\tau$ implies
\[
    \delta
    \leq
    c\left(
        \lambda_{\min}^{-1}
        -
        \lambda_{\max}^{-1}
    \right).
\]
Thus, for $c\leq1$, we have
$a_1\in\mathcal{A}_{\Lambda}$. Moreover,
$\delta\leq ca_0$, and hence
\[
    |\lambda_1-\lambda_0|
    =
    \frac{\delta}{a_0(a_0+\delta)}
    =
    \frac{c\lambda_{\max}\tau}{1+c\tau}.
\]
Since $\tau\leq1$, this gives
\[
    |\lambda_1-\lambda_0|
    \gtrsim
    \lambda_{\max}\tau
    \gtrsim
    \min\left\{
        \Delta_{\Lambda},\,
        \lambda_{\min}\frac{k}{d}
    \right\}.
\]
It remains to construct a prior with low covering complexity that keeps the
two observation laws close in total variation.

Fix an integer $M\geq10$, to be chosen as a sufficiently large
absolute constant, and use the convention
$\mathcal{N}(0,0I_d)=\delta_0$. Consider the ideal signal distribution
\[
    P^{\star}
    :=
    \frac{1}{M}
    \sum_{m=0}^{M-1}
    \mathcal{N}(0,m\delta I_d).
\]
For $i\in\{0,1\}$, let
\[
    Q_i^{\star}
    :=
    P^{\star}*
    \mathcal{N}(0,a_iI_d).
\]
By Gaussian convolution,
\[
    Q_0^{\star}
    =
    \frac{1}{M}
    \sum_{m=0}^{M-1}
    \mathcal{N}(0,(a_0+m\delta)I_d),
\]
whereas
\[
    Q_1^{\star}
    =
    \frac{1}{M}
    \sum_{m=0}^{M-1}
    \mathcal{N}(0,(a_0+(m+1)\delta)I_d).
\]
Hence the two mixtures share $M-1$ of their $M$ Gaussian components
and differ only in their two endpoint components:
\[
    Q_0^{\star}-Q_1^{\star}
    =
    \frac{1}{M}
    \left\{
        \mathcal{N}(0,a_0I_d)
        -
        \mathcal{N}(0,(a_0+M\delta)I_d)
    \right\}.
\]
Therefore
\[
    \|Q_0^{\star}-Q_1^{\star}\|_{\mathrm{TV}}
    \leq
    \frac{1}{M}.
\]
The ideal signal law $P^{\star}$ has the required overlap but unbounded
support. To obtain an admissible prior, we replace each Gaussian component by
a finite point cloud.

\paragraph{Finite-support approximation.}
For each $m\in\{0,\ldots,M-1\}$, choose points
$x_{1,m},\ldots,x_{J,m}\in\mathbb{R}^d$ and define
\[
    \widehat{P}_m
    :=
    \frac{1}{J}
    \sum_{j=1}^{J}\delta_{x_{j,m}}.
\]
Define the approximating prior by
\[
    P
    :=
    \frac{1}{M}
    \sum_{m=0}^{M-1}\widehat{P}_m
    =
    \frac{1}{JM}
    \sum_{m=0}^{M-1}
    \sum_{j=1}^{J}
    \delta_{x_{j,m}}.
\]
Since $P$ has at most $JM$ support points,
\[
    N(\operatorname{supp}(P),r_0)\leq JM.
\]
Thus the covering requirement
\[
    1+\log N(\operatorname{supp}(P),r_0)\leq k
\]
holds whenever
\[
    JM\leq e^{k-1},
    \qquad\text{equivalently}\qquad
    J\leq\frac{e^{k-1}}{M}.
\]

For $a>0$, define the smoothed empirical layer
\[
    \widehat{Q}_{m,a}
    :=
    \widehat{P}_m*
    \mathcal{N}(0,aI_d)
    =
    \frac{1}{J}
    \sum_{j=1}^{J}
    \mathcal{N}(x_{j,m},aI_d),
\]
and let
\[
    G_{m,a}
    :=
    \mathcal{N}(0,(a+m\delta)I_d).
\]
We choose the support points so that $\widehat{Q}_{m,a_i}$ approximates
$G_{m,a_i}$ for every layer $m$ and both $i\in\{0,1\}$.

\paragraph{Single-layer $\chi^2$ approximation.}
Fix $m\in\{0,\ldots,M-1\}$ and $a>0$, and set
$b:=m\delta$. Draw
\[
    X_{1,m},\ldots,X_{J,m}
    \stackrel{\mathrm{iid}}{\sim}
    \mathcal{N}(0,bI_d),
    \qquad
    \widehat{P}_m
    :=
    \frac{1}{J}\sum_{j=1}^{J}\delta_{X_{j,m}}.
\]

\begin{lemma}[Single-layer approximation]
\label{lem:ld-single-layer-chi2}
With the notation above,
\[
    \mathbb{E}
    \left[
        \chi^2(
            \widehat{Q}_{m,a}\|G_{m,a}
        )
    \right]
    =
    \frac{1}{J}
    \left[
        \left(1+\frac{b}{a}\right)^d-1
    \right]
    \leq
    \frac{1}{J}
    \exp\left(\frac{bd}{a}\right).
\]
\end{lemma}

\begin{proof}
Let
\[
    \varphi_a(y-x)
    :=
    (2\pi a)^{-d/2}
    \exp\left(
        -\frac{\|y-x\|^2}{2a}
    \right)
\]
be the density of $\mathcal{N}(x,aI_d)$. The density of
$\widehat{Q}_{m,a}$ is
\[
    \widehat{q}(y)
    =
    \frac{1}{J}
    \sum_{j=1}^{J}
    \varphi_a(y-X_{j,m}),
\]
while the density of $G_{m,a}$ is
$g(y):=\varphi_{a+b}(y)$. Since
$X_{j,m}\sim\mathcal{N}(0,bI_d)$, Gaussian convolution gives
\[
    \mathbb{E}\varphi_a(y-X_{j,m})
    =
    \varphi_{a+b}(y)
    =
    g(y),
\]
and hence $\mathbb{E}\widehat{q}(y)=g(y)$.

By definition,
\[
    \chi^2(
        \widehat{Q}_{m,a}\|G_{m,a}
    )
    =
    \int
    \frac{(\widehat{q}(y)-g(y))^2}{g(y)}
    \,dy
    =
    \int
    \frac{\widehat{q}(y)^2}{g(y)}
    \,dy
    -1.
\]
Expanding $\widehat{q}^2$ and separating diagonal and
off-diagonal terms gives
\[
    \mathbb{E}
    \int
    \frac{\widehat{q}(y)^2}{g(y)}
    \,dy
    =
    \frac{J(J-1)}{J^2}
    +
    \frac{1}{J}\,
    \mathbb{E}I(X_{1,m}),
\]
where
\[
    I(x)
    :=
    \int
    \frac{\varphi_a(y-x)^2}{g(y)}
    \,dy.
\]
Therefore
\[
    \mathbb{E}
    \chi^2(
        \widehat{Q}_{m,a}\|G_{m,a}
    )
    =
    \frac{1}{J}
    \left\{
        \mathbb{E}I(X_{1,m})-1
    \right\}.
\]

A direct Gaussian calculation yields
\[
    I(x)
    =
    \left(
        \frac{(a+b)^2}{a(a+2b)}
    \right)^{d/2}
    \exp\left(
        \frac{\|x\|^2}{a+2b}
    \right).
\]
Since $X_{1,m}\sim\mathcal{N}(0,bI_d)$,
\[
    \mathbb{E}
    \exp\left(
        \frac{\|X_{1,m}\|^2}{a+2b}
    \right)
    =
    \left(
        \frac{a+2b}{a}
    \right)^{d/2}.
\]
Consequently,
\[
    \mathbb{E}I(X_{1,m})
    =
    \left(
        1+\frac{b}{a}
    \right)^d,
\]
and hence
\[
    \mathbb{E}
    \left[
        \chi^2(
            \widehat{Q}_{m,a}\|G_{m,a}
        )
    \right]
    =
    \frac{1}{J}
    \left[
        \left(
            1+\frac{b}{a}
        \right)^d
        -1
    \right].
\]
The final inequality follows from
$(1+u)^d\leq e^{du}$ for $u\geq0$.
\end{proof}

\paragraph{Simultaneous approximation and support control.}
We now choose a deterministic realization of the point clouds that satisfies the
approximation bounds for all layers and both noise variances. For
each $m\in\{0,\ldots,M-1\}$, draw
\[
    X_{1,m},\ldots,X_{J,m}
    \stackrel{\mathrm{iid}}{\sim}
    \mathcal{N}(0,m\delta I_d).
\]
Lemma~\ref{lem:ld-single-layer-chi2} gives, for
$a\in\{a_0,a_1\}$,
\[
    \mathbb{E}
    \left[
        \chi^2(
            \widehat{Q}_{m,a}\|G_{m,a}
        )
    \right]
    \leq
    \frac{1}{J}
    \exp\left(
        \frac{M\delta d}{a_0}
    \right).
\]
Fix a sufficiently small absolute constant $C_0>0$ and a sufficiently
large absolute constant $M$. Then choose $c>0$ sufficiently small,
depending only on $M$, and choose $J$ large enough that
\[
    J
    \geq
    C_1
    \exp\left(
        C_2\frac{M\delta d}{a_0}
    \right),
\]
where $C_1,C_2>0$ are sufficiently large constants depending only on
$C_0$ and $M$. With these choices,
\[
    \mathbb{E}
    \left[
        \chi^2(
            \widehat{Q}_{m,a}\|G_{m,a}
        )
    \right]
    \leq
    \frac{C_0}{8M}
\]
for every $m$ and $a\in\{a_0,a_1\}$. Markov's inequality therefore
gives
\[
    \mathbb{P}
    \left(
        \chi^2(
            \widehat{Q}_{m,a}\|G_{m,a}
        )
        >C_0
    \right)
    \leq
    \frac{1}{8M}.
\]
A union bound over the $2M$ pairs $(m,a)$ shows that, with probability
at least $3/4$,
\[
    \chi^2(
        \widehat{Q}_{m,a}\|G_{m,a}
    )
    \leq C_0
\]
simultaneously for every $m$ and
$a\in\{a_0,a_1\}$. On this event,
\[
    \|
        \widehat{Q}_{m,a}-G_{m,a}
    \|_{\mathrm{TV}}
    \leq
    \sqrt{C_0},
\]
where we used the standard comparison between total variation and
$\chi^2$ divergence.

We must also ensure that the support lies inside $B(0,R)$. Standard
Gaussian concentration and a union bound imply that, with probability
at least $3/4$,
\[
    \max_{0\leq m\leq M-1}
    \max_{1\leq j\leq J}
    \|X_{j,m}\|^2
    \leq
    C M\delta
    \bigl\{
        d+\log(JM)
    \bigr\}.
\]

The chosen variance increment must satisfy both the approximation and
covering-number requirements. Recall that
\[
    \delta
    =
    c a_0\tau,
    \qquad
    \tau
    =
    \min\left\{
        \frac{k}{d},\,
        \frac{\Delta_{\Lambda}}{\lambda_{\min}},\,
        1
    \right\}.
\]
Since $\tau\leq k/d$,
\[
    \frac{M\delta d}{a_0}
    =
    Mc\tau d
    \leq
    Mc k.
\]
Hence, if the absolute constant $c>0$ is sufficiently small, the requirements
\[
    C_1
    \exp\left(
        C_2\frac{M\delta d}{a_0}
    \right)
    \leq J
    \leq
    \frac{e^{k-1}}{M}
\]
are compatible for all sufficiently large $k$. We may therefore
choose an integer $J$ satisfying both inequalities.

The upper bound $JM\leq e^{k-1}$ gives
\[
    1+\log N(
        \operatorname{supp}(P),r_0
    )
    \leq k.
\]
Moreover, $\log(JM)\leq k$, so the support estimate becomes
\[
    \max_{m,j}
    \|X_{j,m}\|^2
    \leq
    CM\delta(d+k)
    \leq
    CMca_0\frac{k}{d}(d+k)
    \lesssim_{\Lambda}
    k,
\]
where we used $\tau\leq k/d$ and
$k\lesssim_{\Lambda}d$. Hence
$\operatorname{supp}(P)\subseteq B(0,R)$ whenever
$R^2\gtrsim_{\Lambda}k$.

The simultaneous approximation event and the support-control event
each have probability at least $3/4$, so their intersection has
positive probability. We can therefore fix a deterministic collection of support points
$\{x_{j,m}\}$ satisfying both bounds.

\paragraph{Total-variation comparison.}
Fix such a deterministic realization and define
\[
    P
    =
    \frac{1}{JM}
    \sum_{m=0}^{M-1}
    \sum_{j=1}^{J}
    \delta_{x_{j,m}}.
\]
For $i\in\{0,1\}$, let
\[
    \mu_i
    :=
    P*
    \mathcal{N}(0,a_iI_d).
\]
For $a=a_0$,
\[
    \mu_0
    =
    \frac{1}{M}
    \sum_{m=0}^{M-1}
    \widehat{Q}_{m,a_0},
    \qquad
    Q_0^\star
    =
    \frac{1}{M}
    \sum_{m=0}^{M-1}
    G_{m,a_0}.
\]
By convexity of total variation,
\[
    \|\mu_0-Q_0^\star\|_{\mathrm{TV}}
    \leq
    \frac{1}{M}
    \sum_{m=0}^{M-1}
    \|
        \widehat{Q}_{m,a_0}-G_{m,a_0}
    \|_{\mathrm{TV}}
    \leq
    \sqrt{C_0}.
\]
Similarly,
\[
    \|\mu_1-Q_1^\star\|_{\mathrm{TV}}
    \leq
    \sqrt{C_0}.
\]
Combining these estimates with
$\|Q_0^\star-Q_1^\star\|_{\mathrm{TV}}\leq1/M$ gives
\[
    \|\mu_0-\mu_1\|_{\mathrm{TV}}
    \leq
    \|\mu_0-Q_0^\star\|_{\mathrm{TV}}
    +
    \|Q_0^\star-Q_1^\star\|_{\mathrm{TV}}
    +
    \|Q_1^\star-\mu_1\|_{\mathrm{TV}}
    \leq
    2\sqrt{C_0}+\frac{1}{M}.
\]
Choosing $C_0>0$ sufficiently small and then $M$ sufficiently large
yields
\[
    \|\mu_0-\mu_1\|_{\mathrm{TV}}
    \leq c_0
\]
for some absolute constant $c_0<1$.

\paragraph{Completion of the construction.}
For bounded $k$, the universal $d^{-1}$ lower bound from
Proposition~\ref{prop:ld-universal-lb} dominates the contribution from covering
complexity. We can therefore restrict to $k\geq k_0$ for a sufficiently large
numerical constant $k_0$.

Recall that
\[
    a_0=\lambda_{\max}^{-1},
    \qquad
    \tau
    =
    \min\left\{
        \frac{k}{d},\,
        \frac{\Delta_{\Lambda}}{\lambda_{\min}},\,
        1
    \right\},
    \qquad
    \delta=ca_0\tau,
\]
and
\[
    a_1=a_0+\delta,
    \qquad
    \lambda_i=a_i^{-1}.
\]
By construction, $a_0,a_1\in\mathcal{A}_{\Lambda}$, and hence
$\lambda_0,\lambda_1\in\Lambda$. Furthermore,
\[
    |\lambda_1-\lambda_0|
    =
    \frac{\delta}{a_0(a_0+\delta)}
    =
    \frac{c\lambda_{\max}\tau}{1+c\tau}
    \gtrsim
    \min\left\{
        \Delta_{\Lambda},\,
        \lambda_{\min}\frac{k}{d}
    \right\}.
\]

The deterministic point cloud constructed above satisfies both
\[
    1+\log
    N(\operatorname{supp}(P),r_0)
    \leq k
\]
and $\operatorname{supp}(P)\subseteq B(0,R)$, so
$P\in\mathcal{P}_{\mathrm{LD}}(k,R)$. The corresponding observation
laws satisfy
\[
    \|
        P_{\lambda_0}-P_{\lambda_1}
    \|_{\mathrm{TV}}
    \leq c_0<1.
\]
This proves Lemma~\ref{lem:ld-indistinguishable-pair}.

Applying the standard two-point lower bound to the pair supplied by
Lemma~\ref{lem:ld-indistinguishable-pair} yields
\[
    \mathcal{R}^{\star}_{\mathrm{LD}}(d,k,R;\Lambda)
    \gtrsim
    \min\left\{
        \Delta_{\Lambda}^2,\,
        \lambda_{\min}^2\frac{k^2}{d^2}
    \right\}.
\]
Together with Proposition~\ref{prop:ld-universal-lb}, which gives
\[
    \mathcal{R}^{\star}_{\mathrm{LD}}(d,k,R;\Lambda)
    \gtrsim
    \min\left\{
        \Delta_{\Lambda}^2,\,
        \frac{\lambda_{\min}^2}{d}
    \right\},
\]
we obtain
\[
    \mathcal{R}^{\star}_{\mathrm{LD}}(d,k,R;\Lambda)
    \gtrsim
    \min\left\{
        \Delta_{\Lambda}^2,\,
        \lambda_{\min}^2
        \left(
            \frac{1}{d}
            +
            \frac{k^2}{d^2}
        \right)
    \right\}.
\]
Here we used the elementary equivalence
\[
    \max\{\min(A,u),\min(A,v)\}
    \asymp
    \min\{A,u+v\}
\]
for nonnegative $A,u,v$.
\subsection{MLE upper bound}
\label{app:ld-mle-upper}

We now prove the MLE upper bound. Fix
$P\in\mathcal{P}_{\mathrm{LD}}(k,R)$ and let
$\lambda^\star\in\Lambda$ denote the true precision, so that
$Y\sim P_{\lambda^\star}$. Recall
\[
    \ell(\lambda\mid y)
    :=
    -\log p_{\lambda}(y),
    \qquad
    D_{\lambda}(y)
    :=
    \mathbb{E}_{q_{\lambda}(\cdot\mid y)}
    \|X-y\|^2,
\]
and the score identity
\[
    2\ell'(\lambda\mid y)
    =
    -\frac{d}{\lambda}
    +
    D_{\lambda}(y).
\]

Low intrinsic dimensionality bounds the posterior quadratic error uniformly
over candidate precisions.

\begin{lemma}[Uniform posterior-energy approximation]
\label{lem:ld-posterior-energy}
Let $P\in\mathcal{P}_{\mathrm{LD}}(k,R)$,
$\lambda^\star\in\Lambda$, and $Y\sim P_{\lambda^\star}$. For every
$\eta\in(0,1/2)$, with probability at least $1-\eta$,
\[
    \sup_{\lambda\in\Lambda}
    \left|
        D_{\lambda}(Y)
        -
        \frac{d}{\lambda^\star}
    \right|
    \leq
    B_{\eta},
\]
where
\[
    B_{\eta}
    :=
    \frac{C}{\lambda_{\min}}
    \left(
        \sqrt{
            d\log\frac{2}{\eta}
        }
        +
        k
        +
        \log\frac{2}{\eta}
    \right)
\]
for a numerical constant $C>0$.
\end{lemma}

\begin{proof}
Fix a candidate precision $\lambda\in\Lambda$ and write
\[
    \widetilde{\sigma}
    :=
    (\lambda^\star)^{-1/2},
    \qquad
    \sigma:=\lambda^{-1/2}.
\]
Proposition~F.2 of
\citet{kadkhodaie2026blinddenoising}, applied with
$\sigma_T=\lambda_{\max}^{-1/2}$,
$\sigma_0=\lambda_{\min}^{-1/2}$, and $\ell=2$, gives exactly the
required bound for $D_{\lambda}(Y)$. In the present notation, with
probability at least $1-\delta$,
\[
    \left|
        D_{\lambda}(Y)
        -
        \frac{d}{\lambda^\star}
    \right|
    \leq
    C
    \left[
        \frac{1}{\lambda^\star}
        \sqrt{
            d\log\frac{1}{\delta}
        }
        +
        \left(
            \frac{1}{\lambda^\star}
            +
            \frac{1}{\lambda}
        \right)
        \left(
            k+\log\frac{1}{\delta}
        \right)
    \right].
\]
Since $\lambda,\lambda^\star\geq\lambda_{\min}$, the right-hand side
is at most
\[
    \frac{C}{\lambda_{\min}}
    \left(
        \sqrt{
            d\log\frac{1}{\delta}
        }
        +
        k
        +
        \log\frac{1}{\delta}
    \right).
\]

We apply this estimate only at the two endpoints
$\lambda_{\min}$ and $\lambda_{\max}$, each with
$\delta=\eta/2$. A union bound shows that, with probability at least
$1-\eta$,
\[
    \left|
        D_{\lambda_{\min}}(Y)
        -
        \frac{d}{\lambda^\star}
    \right|
    \leq B_{\eta},
    \qquad
    \left|
        D_{\lambda_{\max}}(Y)
        -
        \frac{d}{\lambda^\star}
    \right|
    \leq B_{\eta}.
\]

Monotonicity extends the endpoint bounds to the entire interval.
Differentiating the posterior expectation with respect to
$\lambda$ gives
\[
    \partial_{\lambda}D_{\lambda}(y)
    =
    -\frac{1}{2}
    \operatorname{Var}_{q_{\lambda}(\cdot\mid y)}
    \left(
        \|X-y\|^2
    \right)
    \leq0.
\]
Hence $D_{\lambda}(y)$ is nonincreasing in $\lambda$, and therefore
\[
    D_{\lambda_{\max}}(Y)
    \leq
    D_{\lambda}(Y)
    \leq
    D_{\lambda_{\min}}(Y)
    \qquad
    \text{for every }\lambda\in\Lambda.
\]
Since both endpoint values lie within $B_{\eta}$ of
$d/\lambda^\star$, the same holds uniformly over
$\lambda\in\Lambda$.
\end{proof}

\begin{proposition}[MLE upper bound]
\label{prop:ld-mle-upper}
Let $P\in\mathcal{P}_{\mathrm{LD}}(k,R)$ and
\[
    \widehat{\lambda}_{\mathrm{MLE},P}(Y)
    \in
    \argmin_{\lambda\in\Lambda}
    \ell(\lambda\mid Y).
\]
For every $\lambda^\star\in\Lambda$ and
$\eta\in(0,1/2)$, with probability at least $1-\eta$ under
$Y\sim P_{\lambda^\star}$,
\[
    \left|
        \widehat{\lambda}_{\mathrm{MLE},P}
        -
        \lambda^\star
    \right|
    \leq
    C
    \frac{\lambda_{\max}^2}{\lambda_{\min}}
    \left(
        \sqrt{
            \frac{\log(2/\eta)}{d}
        }
        +
        \frac{
            k+\log(2/\eta)
        }{d}
    \right).
\]
Consequently,
\[
    \sup_{\lambda^\star\in\Lambda}
    \mathbb{E}_{\lambda^\star,P}
    \left[
        \left(
            \widehat{\lambda}_{\mathrm{MLE},P}
            -
            \lambda^\star
        \right)^2
    \right]
    \lesssim
    \min\left\{
        \Delta_{\Lambda}^2,\,
        \frac{\lambda_{\max}^4}{\lambda_{\min}^2}
        \left[
            \frac{\log(2d)}{d}
            +
            \frac{
                (k+\log(2d))^2
            }{d^2}
        \right]
    \right\}.
\]
\end{proposition}

\begin{proof}
Define the Gaussian reference excess likelihood
\[
    \Psi_{\lambda^\star}(\lambda)
    :=
    \frac{d}{2}
    \left[
        \frac{\lambda}{\lambda^\star}
        -1
        -
        \log\left(
            \frac{\lambda}{\lambda^\star}
        \right)
    \right].
\]
It satisfies
\[
    \Psi_{\lambda^\star}(\lambda^\star)
    =
    \Psi'_{\lambda^\star}(\lambda^\star)
    =
    0,
\]
and
\[
    \Psi'_{\lambda^\star}(\lambda)
    =
    \frac{1}{2}
    \left(
        -\frac{d}{\lambda}
        +
        \frac{d}{\lambda^\star}
    \right),
    \qquad
    \Psi''_{\lambda^\star}(\lambda)
    =
    \frac{d}{2\lambda^2}.
\]
Since $\lambda\leq\lambda_{\max}$, Taylor's theorem gives
\[
    \Psi_{\lambda^\star}(\lambda)
    \geq
    \frac{d}{4\lambda_{\max}^2}
    (\lambda-\lambda^\star)^2.
\]

Let $B_{\eta}$ be as in
Lemma~\ref{lem:ld-posterior-energy}, and work on the event supplied by
that lemma. By the score identity and the fundamental theorem of
calculus,
\begin{align*}
    &
    \ell(\lambda\mid Y)
    -
    \ell(\lambda^\star\mid Y)
    -
    \Psi_{\lambda^\star}(\lambda)
    \\
    &\qquad
    =
    \frac{1}{2}
    \int_{\lambda^\star}^{\lambda}
    \left(
        D_u(Y)
        -
        \frac{d}{\lambda^\star}
    \right)\,du.
\end{align*}
Therefore
\[
    \left|
        \ell(\lambda\mid Y)
        -
        \ell(\lambda^\star\mid Y)
        -
        \Psi_{\lambda^\star}(\lambda)
    \right|
    \leq
    \frac{B_{\eta}}{2}
    |\lambda-\lambda^\star|.
\]
Combining this with the quadratic lower bound yields
\[
    \ell(\lambda\mid Y)
    -
    \ell(\lambda^\star\mid Y)
    \geq
    \frac{d}{4\lambda_{\max}^2}
    (\lambda-\lambda^\star)^2
    -
    \frac{B_{\eta}}{2}
    |\lambda-\lambda^\star|.
\]

By optimality of the constrained MLE,
\[
    \ell(
        \widehat{\lambda}_{\mathrm{MLE},P}
        \mid Y
    )
    \leq
    \ell(\lambda^\star\mid Y).
\]
Substituting
$\lambda=\widehat{\lambda}_{\mathrm{MLE},P}$ into the preceding
inequality gives
\[
    0
    \geq
    \frac{d}{4\lambda_{\max}^2}
    \left|
        \widehat{\lambda}_{\mathrm{MLE},P}
        -
        \lambda^\star
    \right|^2
    -
    \frac{B_{\eta}}{2}
    \left|
        \widehat{\lambda}_{\mathrm{MLE},P}
        -
        \lambda^\star
    \right|.
\]
Hence
\[
    \left|
        \widehat{\lambda}_{\mathrm{MLE},P}
        -
        \lambda^\star
    \right|
    \leq
    C
    \frac{\lambda_{\max}^2}{d}
    B_{\eta},
\]
which, after substituting the expression for $B_{\eta}$, gives
\[
    \left|
        \widehat{\lambda}_{\mathrm{MLE},P}
        -
        \lambda^\star
    \right|
    \leq
    C
    \frac{\lambda_{\max}^2}{\lambda_{\min}}
    \left(
        \sqrt{
            \frac{\log(2/\eta)}{d}
        }
        +
        \frac{
            k+\log(2/\eta)
        }{d}
    \right).
\]

To bound the mean-squared error using this probability estimate,
assume first that $d\geq2$ and take $\eta=d^{-3}$. On the event above,
\[
    \left|
        \widehat{\lambda}_{\mathrm{MLE},P}
        -
        \lambda^\star
    \right|^2
    \leq
    C
    \frac{\lambda_{\max}^4}{\lambda_{\min}^2}
    \left[
        \frac{\log(2d)}{d}
        +
        \frac{
            (k+\log(2d))^2
        }{d^2}
    \right].
\]
On the complementary event, both
$\widehat{\lambda}_{\mathrm{MLE},P}$ and $\lambda^\star$ lie in
$\Lambda$, so deterministically
\[
    \left|
        \widehat{\lambda}_{\mathrm{MLE},P}
        -
        \lambda^\star
    \right|^2
    \leq
    \Delta_{\Lambda}^2.
\]
Taking expectations gives
\begin{align*}
    &
    \mathbb{E}_{\lambda^\star,P}
    \left[
        \left(
            \widehat{\lambda}_{\mathrm{MLE},P}
            -
            \lambda^\star
        \right)^2
    \right]
    \\
    &\qquad
    \leq
    C
    \frac{\lambda_{\max}^4}{\lambda_{\min}^2}
    \left[
        \frac{\log(2d)}{d}
        +
        \frac{
            (k+\log(2d))^2
        }{d^2}
    \right]
    +
    \Delta_{\Lambda}^2d^{-3}.
\end{align*}
Combining this estimate with the deterministic diameter bound gives
the claimed minimum with $\Delta_{\Lambda}^2$. The bounded-dimensional
case follows after enlarging the numerical constant. The result is
uniform over $\lambda^\star\in\Lambda$ and over fixed
$P\in\mathcal{P}_{\mathrm{LD}}(k,R)$.
\end{proof}

\subsection{Norm-based estimators}
\label{app:norm-estimators}

We give the norm-statistic calculation used in Section~\ref{sec:upper-techniques}.
This is not an impossibility result for all estimators depending on
$\|Y\|^2$. The calculation concerns the standard energy-corrected norm
statistic: its variance cannot be bounded uniformly in terms of the
covering-complexity parameter $k$ alone.

Write
\[
a:=\lambda^{-1},
\qquad
Y=X+\sqrt{a}Z,
\qquad
Z\sim\mathcal N(0,I_d),
\]
with $Z$ independent of $X$. The raw norm statistic
\[
\widetilde a_{\mathrm{raw}}
:=
\frac{\|Y\|^2}{d}
\]
satisfies
\[
\mathbb E_{\lambda,P}\widetilde a_{\mathrm{raw}}
=
a+\frac{\mathbb E_P\|X\|^2}{d}.
\]
Thus the raw norm confounds signal energy with noise energy. For example, if
$P=\delta_\theta$ with $\|\theta\|=r\le R$, then $P$ has covering complexity
one but
\[
\mathbb E_{\lambda,P}\widetilde a_{\mathrm{raw}}-a
=
\frac{r^2}{d}.
\]

Since $P$ is known, this bias can be removed by the energy-corrected statistic
\[
\widetilde a_P(Y)
:=
\frac{\|Y\|^2-\mathbb E_P\|X\|^2}{d},
\qquad
\mathbb E_{\lambda,P}\widetilde a_P(Y)=a.
\]
By Lemma~\ref{lem:radial-variance} with $m=0$,
\[
\operatorname{Var}_{\lambda,P}\!\left(\widetilde a_P(Y)\right)
=
\frac{
\operatorname{Var}_P(\|X\|^2)
+4a\,\mathbb E_P\|X\|^2
+2a^2d
}{d^2}.
\]
The covering-number condition gives no bound on the first term using $k$
alone. For example, take
\[
P=\frac12\delta_0+\frac12\delta_u,
\qquad
\|u\|=r\le R.
\]
This distribution has constant covering complexity, whereas
\[
\mathbb E_P\|X\|^2=\frac{r^2}{2},
\qquad
\operatorname{Var}_P(\|X\|^2)=\frac{r^4}{4}.
\]
Therefore
\[
\operatorname{Var}_{\lambda,P}\!\left(\widetilde a_P(Y)\right)
=
\frac{r^4}{4d^2}
+\frac{2ar^2}{d^2}
+\frac{2a^2}{d}.
\]
The covering assumption allows $r$ to be as large as the support
radius $R$. Taking $r\asymp R$, the first term is of order
\[
    \frac{R^4}{d^2}.
\]
Since the class allows $R^2\gg k$, this can be much larger than the
intrinsic-dimensional scale $k^2/d^2$. Thus the covering condition
alone does not provide a uniform bound on radial fluctuations in
terms of $k$.

Finally, the map $a\mapsto a^{-1}$ has derivative bounded above and
below on the compact interval
$[\lambda_{\max}^{-1},\lambda_{\min}^{-1}]$. Thus fluctuations in a
variance estimator translate, up to factors determined by
$\lambda_{\min}$ and $\lambda_{\max}$, into fluctuations of the same
local order for the corresponding precision estimator.

The centered norm statistic therefore gives no uniform upper bound based
only on covering complexity. The fixed-$P$ MLE uses the full signal law and
satisfies the bound proved above.

\section{Proofs for Strong Log-Concavity}
\label{app:slc-proofs}

For a fixed $P\in\mathcal P_{\mathrm{SLC}}(\alpha)$, let
$\mu_P:=\mathbb E_PX$. Replacing $(X,Y)$ by
$(X-\mu_P,Y-\mu_P)$ preserves the precision parameter and the strong
log-concavity constant. We can therefore center the signal throughout this section, taking
$\mathbb E_PX=0$.

For $P\in\mathcal P_{\mathrm{SLC}}(\alpha)$, write
\[
\Lambda=[\lambda_{\min},\lambda_{\max}],
\qquad
\Delta_\Lambda:=\lambda_{\max}-\lambda_{\min}.
\]
With $P$ centered so that $\mathbb E_P X=0$, the observation law at precision
$\lambda\in\Lambda$ is
\[
P_\lambda
=
P*N(0,\lambda^{-1}I_d).
\]
We also use
\[
\rho_{\alpha,d}(\Lambda)
:=
\min\left\{
\Delta_\Lambda^2,\,
\frac{(1+\alpha^{-1})^2}{d}
\right\}.
\]

\subsection{Gaussian lower bound}
\label{app:slc-lower}

For the lower bound in Theorem~\ref{thm:slc}, consider the centered
Gaussian signal law
\[
P_\alpha:=N(0,\alpha^{-1}I_d).
\]
Its potential is $V(x)=\alpha\|x\|^2/2$, so
$\nabla^2V(x)=\alpha I_d$ and hence
$P_\alpha\in\mathcal P_{\mathrm{SLC}}(\alpha)$. Under this fixed prior,
\[
Y
\sim
N\!\left(
0,\,
(\alpha^{-1}+\lambda^{-1})I_d
\right).
\]
It suffices to bound from below the risk
$\mathcal{R}(P_\alpha;\Lambda)$.

Let
\[
\lambda_{\mathrm{mid}}
:=
\frac{\lambda_{\min}+\lambda_{\max}}{2},
\]
and let $h>0$ be chosen below. Set
\[
\lambda_0:=\lambda_{\mathrm{mid}}-h,
\qquad
\lambda_1:=\lambda_{\mathrm{mid}}+h,
\]
and write
\[
\mu_i
:=
N(0,v_iI_d),
\qquad
v_i:=\alpha^{-1}+\lambda_i^{-1},
\qquad i\in\{0,1\}.
\]
We will choose $h$ small enough that $\lambda_0,\lambda_1\in\Lambda$.

By Lemma~\ref{lem:two-point},
\[
R(P_\alpha;\Lambda)
\ge
\frac{|\lambda_1-\lambda_0|^2}{8}
\left(
1-\|\mu_0-\mu_1\|_{\mathrm{TV}}
\right).
\]
We choose $h$ to keep the two Gaussian laws at total variation distance
bounded away from one.

For centered Gaussians with scalar covariance matrices,
\[
\mathrm{KL}(\mu_0\|\mu_1)
=
\frac d2
\left(
\frac{v_0}{v_1}
-1
-\log\frac{v_0}{v_1}
\right).
\]
Since $\lambda_0<\lambda_1$,
\[
0
\le
\frac{v_0}{v_1}-1
=
\frac{\lambda_0^{-1}-\lambda_1^{-1}}
{\alpha^{-1}+\lambda_1^{-1}}
=
\frac{2h}{\lambda_0(1+\lambda_1/\alpha)}
\le
C_\Lambda
\frac{h}{1+\alpha^{-1}}.
\]
Hence, provided $h\le c_\Lambda(1+\alpha^{-1})$ for a sufficiently small
$c_\Lambda>0$, we have $1\le v_0/v_1\le 3/2$. Using
$r-1-\log r\le C(r-1)^2$ on this interval gives
\[
\mathrm{KL}(\mu_0\|\mu_1)
\le
C_\Lambda
\frac{dh^2}{(1+\alpha^{-1})^2}.
\]

We now choose
\[
h
:=
c_\Lambda
\min\left\{
\Delta_\Lambda,\,
\frac{1+\alpha^{-1}}{\sqrt d}
\right\},
\]
where $c_\Lambda>0$ is sufficiently small. This choice ensures both
$\lambda_0,\lambda_1\in\Lambda$ and
\[
\mathrm{KL}(\mu_0\|\mu_1)\le \frac18.
\]
Pinsker's inequality then gives
\[
\|\mu_0-\mu_1\|_{\mathrm{TV}}
\le
\sqrt{\frac12\mathrm{KL}(\mu_0\|\mu_1)}
\le
\frac14.
\]
The two-point lower bound therefore yields
\[
R(P_\alpha;\Lambda)
\gtrsim
h^2
\asymp_\Lambda
\min\left\{
\Delta_\Lambda^2,\,
\frac{(1+\alpha^{-1})^2}{d}
\right\}
=
\rho_{\alpha,d}(\Lambda).
\]
Since $P_\alpha\in\mathcal P_{\mathrm{SLC}}(\alpha)$,
\[
\mathcal{R}^\star_{\mathrm{SLC}}(d,\alpha;\Lambda)
\ge
R(P_\alpha;\Lambda)
\gtrsim_\Lambda
\rho_{\alpha,d}(\Lambda),
\]
which proves the lower bound in Theorem~\ref{thm:slc}.

\subsection{Norm-estimator upper bound}
\label{app:slc-norm}

The matching upper bound follows from the $P$-centered norm estimator. By Lemma~\ref{lem:slc-moments},
\[
    \mathbb{E}_P\|X-\mu_P\|^2
    \leq
    \frac{d}{\alpha},
    \qquad
    \operatorname{Var}_P
    \bigl(\|X-\mu_P\|^2\bigr)
    \leq
    \frac{4d}{\alpha^2}.
\]

Write $a=\lambda^{-1}$ and let
\[
A_\Lambda
:=
[\lambda_{\max}^{-1},\lambda_{\min}^{-1}]
\]
denote the admissible noise-variance interval. Define the $P$-centered statistic
\[
\widetilde a_P(Y)
:=
\frac{
    \|Y-\mu_P\|^2
    -
    \mathbb E_P\|X-\mu_P\|^2
}{d}.
\]
Define its projection by
\[
\widehat a_P(Y)
:=
\Pi_{A_\Lambda}\bigl(\widetilde a_P(Y)\bigr),
\]
and define the corresponding precision estimator by
\[
\widehat\lambda_{\mathrm{norm},P}(Y)
:=
\widehat a_P(Y)^{-1}.
\]

Fix $\lambda\in\Lambda$ and write $a=\lambda^{-1}$. By
Lemma~\ref{lem:radial-variance}, applied with $m=\mu_P$,
\[
    \mathbb{E}_{\lambda,P}\widetilde a_P(Y)=a
\]
and
\[
    \operatorname{Var}_{\lambda,P}
    \bigl(\|Y-\mu_P\|^2\bigr)
    =
    \operatorname{Var}_P
    \bigl(\|X-\mu_P\|^2\bigr)
    +
    4a\,\mathbb{E}_P\|X-\mu_P\|^2
    +
    2a^2d.
\]
Using Lemma~\ref{lem:slc-moments} and
$a\leq\lambda_{\min}^{-1}$,
\begin{align*}
\mathbb E_{\lambda,P}
\left[
\bigl(\widetilde a_P(Y)-a\bigr)^2
\right]
=
\frac{
    \operatorname{Var}_{\lambda,P}
    \bigl(\|Y-\mu_P\|^2\bigr)
}{d^2} \\
&\le
\frac1d
\left(
\frac4{\alpha^2}
+
\frac4{\alpha\lambda_{\min}}
+
\frac2{\lambda_{\min}^2}
\right) \\
&\lesssim_\Lambda
\frac{(1+\alpha^{-1})^2}{d}.
\end{align*}

Since $a\in A_\Lambda$, Euclidean projection cannot increase the estimation error:
\[
|\widehat a_P(Y)-a|
\le
|\widetilde a_P(Y)-a|.
\]
Furthermore, both $a$ and $\widehat a_P(Y)$ lie in $A_\Lambda$, and hence
\[
\left|
\widehat\lambda_{\mathrm{norm},P}(Y)-\lambda
\right|
=
\left|
\frac1{\widehat a_P(Y)}-\frac1a
\right|
=
\frac{|\widehat a_P(Y)-a|}
{a\widehat a_P(Y)}
\le
\lambda_{\max}^2
|\widehat a_P(Y)-a|.
\]
Therefore
\[
\sup_{\lambda\in\Lambda}
\mathbb E_{\lambda,P}
\left[
\bigl(
\widehat\lambda_{\mathrm{norm},P}(Y)-\lambda
\bigr)^2
\right]
\lesssim_\Lambda
\frac{(1+\alpha^{-1})^2}{d}.
\]
Because both $\widehat\lambda_{\mathrm{norm},P}(Y)$ and $\lambda$ belong to
$\Lambda$, their squared distance is bounded by
$\Delta_\Lambda^2$. Combining the two bounds gives
\[
\sup_{\lambda\in\Lambda}
\mathbb E_{\lambda,P}
\left[
\bigl(
\widehat\lambda_{\mathrm{norm},P}(Y)-\lambda
\bigr)^2
\right]
\lesssim_\Lambda
\min\left\{
\Delta_\Lambda^2,\,
\frac{(1+\alpha^{-1})^2}{d}
\right\}
=
\rho_{\alpha,d}(\Lambda).
\]
Since this holds for every fixed
$P\in\mathcal P_{\mathrm{SLC}}(\alpha)$,
\[
\mathcal{R}^\star_{\mathrm{SLC}}(d,\alpha;\Lambda)
\lesssim_\Lambda
\rho_{\alpha,d}(\Lambda).
\]
Together with Appendix~\ref{app:slc-lower}, this proves the minimax-rate statement in
Theorem~\ref{thm:slc}.

\subsection{R\'enyi separation for noisy marginals}
\label{app:slc-renyi}

To analyze the MLE, we first bound the separation of the noisy marginals.
Fix a true precision
$\lambda^\star\in\Lambda$ and suppose
\[
Y\sim P_{\lambda^\star}.
\]
Define
\[
s_{\alpha,\Lambda}
:=
\alpha^{-1}+\lambda_{\min}^{-1}.
\]
Since $\Lambda$ is fixed,
$s_{\alpha,\Lambda}\asymp_\Lambda 1+\alpha^{-1}$.

For $a>0$, write
\[
\mu_a
:=
P*N(0,aI_d),
\]
and denote its density by $p_a$. Gaussian convolution preserves strong log-concavity: if $P$ is
$\alpha$-strongly log-concave, then $\mu_a$ is
$(\alpha^{-1}+a)^{-1}$-strongly log-concave; see
\cite{saumard2014logconcavity}. Hence, for every
$a\in[\lambda_{\max}^{-1},\lambda_{\min}^{-1}]$,
\[
-\nabla^2\log p_a(y)
\succeq
\frac1{\alpha^{-1}+a}I_d
\succeq
\frac1{s_{\alpha,\Lambda}}I_d.
\]
Thus every noisy marginal in the parameter interval is
$s_{\alpha,\Lambda}^{-1}$-strongly log-concave.

Set
\[
a:=\lambda^{-1},
\qquad
a^\star:=(\lambda^\star)^{-1}.
\]
Define the Hellinger midpoint distribution
\[
\nu(dy)
:=
\frac{\sqrt{p_a(y)p_{a^\star}(y)}}
{\int\sqrt{p_a(z)p_{a^\star}(z)}\,dz}\,dy.
\]
By direct calculation,
\[
D_{1/2}(\mu_{a^\star}\|\mu_a)
=
\mathrm{KL}(\nu\|\mu_{a^\star})
+
\mathrm{KL}(\nu\|\mu_a).
\]
Since both $\mu_{a^\star}$ and $\mu_a$ are
$s_{\alpha,\Lambda}^{-1}$-strongly log-concave, Talagrand's transportation
inequality gives
\begin{align*}
D_{1/2}(\mu_{a^\star}\|\mu_a)
&\ge
\frac1{2s_{\alpha,\Lambda}}
\left\{
W_2^2(\nu,\mu_{a^\star})
+
W_2^2(\nu,\mu_a)
\right\} \\
&\ge
\frac1{4s_{\alpha,\Lambda}}
W_2^2(\mu_{a^\star},\mu_a),
\end{align*}
where the second inequality follows from the triangle inequality for $W_2$
and $u^2+v^2\ge (u+v)^2/2$.

Let
\[
\Sigma:=\operatorname{Cov}_P(X),
\]
and denote its eigenvalues by $\tau_1,\ldots,\tau_d$. The Poincar\'e
inequality gives
\[
0\le \tau_i\le \alpha^{-1},
\qquad i=1,\ldots,d.
\]
Moreover,
\[
\operatorname{Cov}(\mu_a)=\Sigma+aI_d,
\qquad
\operatorname{Cov}(\mu_{a^\star})=\Sigma+a^\star I_d.
\]
The covariance lower bound for the \(2\)-Wasserstein distance \citep{gelbrich1990formula} therefore yields
\begin{align*}
W_2^2(\mu_{a^\star},\mu_a)
&\ge
\sum_{i=1}^d
\left(
\sqrt{\tau_i+a^\star}
-
\sqrt{\tau_i+a}
\right)^2 \\
&=
(a-a^\star)^2
\sum_{i=1}^d
\frac1{
\left(
\sqrt{\tau_i+a^\star}
+
\sqrt{\tau_i+a}
\right)^2
}.
\end{align*}
Since
$\tau_i+a,\tau_i+a^\star\le s_{\alpha,\Lambda}$,
\[
W_2^2(\mu_{a^\star},\mu_a)
\ge
\frac{d(a-a^\star)^2}{4s_{\alpha,\Lambda}}.
\]
Combining the preceding bounds gives
\[
D_{1/2}(\mu_{a^\star}\|\mu_a)
\ge
\frac{d(a-a^\star)^2}
{16s_{\alpha,\Lambda}^2}.
\]
Finally,
\[
|a-a^\star|
=
\frac{|\lambda-\lambda^\star|}
{\lambda\lambda^\star}
\ge
\frac{|\lambda-\lambda^\star|}
{\lambda_{\max}^2},
\]
and hence
\[
D_{1/2}
\bigl(
P_{\lambda^\star}\|P_\lambda
\bigr)
\ge
c_\Lambda
\frac{d(\lambda-\lambda^\star)^2}
{s_{\alpha,\Lambda}^2}.
\tag{C.1}
\label{eq:slc-renyi-separation}
\]
This bound quantifies the R\'enyi separation needed to localize the MLE.

\subsection{Uniform likelihood-ratio control}
\label{app:slc-lr}

We use the separation bound \eqref{eq:slc-renyi-separation} to control
likelihood ratios uniformly over candidate precisions.

Recall that
\[
\ell(\lambda\mid y)
:=
-\log p_\lambda(y).
\]
For a fixed candidate $\lambda\in\Lambda$ and any $u\ge0$,
\begin{align*}
&\mathbb P_{\lambda^\star}
\left(
\ell(\lambda\mid Y)
-
\ell(\lambda^\star\mid Y)
\le u
\right) \\
&\qquad=
\mathbb P_{\lambda^\star}
\left(
\frac{p_\lambda(Y)}
{p_{\lambda^\star}(Y)}
\ge e^{-u}
\right).
\end{align*}
Applying Markov's inequality to the square root of the likelihood ratio gives
\begin{align*}
\mathbb P_{\lambda^\star}
\left(
\ell(\lambda\mid Y)
-
\ell(\lambda^\star\mid Y)
\le u
\right)
&\le
e^{u/2}
\mathbb E_{\lambda^\star}
\left[
\left(
\frac{p_\lambda(Y)}
{p_{\lambda^\star}(Y)}
\right)^{1/2}
\right] \\
&=
\exp\left\{
\frac u2
-
\frac12
D_{1/2}
\bigl(
P_{\lambda^\star}\|P_\lambda
\bigr)
\right\}.
\end{align*}
Using \eqref{eq:slc-renyi-separation} and adjusting $c_\Lambda$,
\[
\mathbb P_{\lambda^\star}
\left(
\ell(\lambda\mid Y)
-
\ell(\lambda^\star\mid Y)
\le u
\right)
\le
\exp\left\{
\frac u2
-
c_\Lambda
\frac{d(\lambda-\lambda^\star)^2}
{s_{\alpha,\Lambda}^2}
\right\}.
\tag{C.2}
\label{eq:fixed-lr}
\]

It remains to make this bound uniform over $\lambda\in\Lambda$. By the score
identity from Section~\ref{sec:background},
\[
\partial_\lambda\ell(\lambda\mid y)
=
-\frac{d}{2\lambda}
+
\frac12D_\lambda(y),
\]
where
\[
D_\lambda(y)
:=
\mathbb E_{q_\lambda(\cdot\mid y)}
\|X-y\|^2.
\]
We first show that
\[
D_\lambda(y)
\le
\mathbb E_P\|X-y\|^2.
\tag{C.3}
\label{eq:posterior-energy-upper}
\]
Let
\[
U:=\|X-y\|^2,
\qquad
w(U):=e^{-\lambda U/2}.
\]
Then
\[
D_\lambda(y)
=
\frac{\mathbb E_P[Uw(U)]}
{\mathbb E_P[w(U)]}.
\]
If $U'$ is an independent copy of $U$, then
\[
2\operatorname{Cov}(U,w(U))
=
\mathbb E
\left[
(U-U')
\bigl(
w(U)-w(U')
\bigr)
\right]
\le0,
\]
because $w$ is decreasing. Hence
\[
\mathbb E_P[Uw(U)]
\le
\mathbb E_P[U]\,
\mathbb E_P[w(U)],
\]
which proves \eqref{eq:posterior-energy-upper}.

Since $\mathbb E_PX=0$ and
$\mathbb E_P\|X\|^2\le d/\alpha$,
\[
D_\lambda(y)
\le
\mathbb E_P\|X-y\|^2
=
\mathbb E_P\|X\|^2+\|y\|^2
\le
\frac d\alpha+\|y\|^2.
\]
Consequently,
\[
\sup_{\lambda\in\Lambda}
|\partial_\lambda\ell(\lambda\mid y)|
\le
C_\Lambda
\left\{
s_{\alpha,\Lambda}d+\|y\|^2
\right\}.
\tag{C.4}
\label{eq:likelihood-derivative}
\]

Under the true precision $\lambda^\star$,
the noisy marginal $P_{\lambda^\star}$ is
$s_{\alpha,\Lambda}^{-1}$-strongly log-concave. Moreover,
\[
\mathbb E_{\lambda^\star}\|Y\|^2
=
\mathbb E_P\|X\|^2+\frac d{\lambda^\star}
\le
ds_{\alpha,\Lambda}.
\]
Standard concentration for strongly log-concave measures, applied to the
$1$-Lipschitz map $y\mapsto\|y\|$, therefore gives, for every $t\ge1$,
\[
\mathbb P_{\lambda^\star}
\left(
\|Y\|^2
>
C_\Lambda
s_{\alpha,\Lambda}(d+t)
\right)
\le
e^{-t}.
\tag{C.5}
\label{eq:y-norm-concentration}
\]
On the complementary event, \eqref{eq:likelihood-derivative} implies
\[
\sup_{\lambda\in\Lambda}
|\partial_\lambda\ell(\lambda\mid Y)|
\le
L_t,
\qquad
L_t
:=
C_\Lambda
s_{\alpha,\Lambda}(d+t).
\]

Set
\[
\varepsilon_t:=L_t^{-1}
\]
and choose an $\varepsilon_t$-net
$\{\lambda_1,\ldots,\lambda_{N_t}\}$ of $\Lambda$. Since $\Lambda$ is a
fixed compact interval,
\[
N_t
\le
1+\frac{\Delta_\Lambda}{\varepsilon_t}
\le
C_\Lambda
s_{\alpha,\Lambda}(d+t).
\tag{C.6}
\label{eq:net-size}
\]
On the event in \eqref{eq:y-norm-concentration}, for every
$\lambda\in\Lambda$ and a nearest net point $\lambda_j$,
\[
|\ell(\lambda\mid Y)-\ell(\lambda_j\mid Y)|
\le
L_t\varepsilon_t
=
1.
\]

Suppose there exists $\lambda\in\Lambda$ such that
\[
|\lambda-\lambda^\star|
\ge
r+\varepsilon_t
\]
and
\[
\ell(\lambda\mid Y)
\le
\ell(\lambda^\star\mid Y).
\]
Its nearest net point then satisfies
\[
|\lambda_j-\lambda^\star|
\ge r
\]
and
\[
\ell(\lambda_j\mid Y)
-
\ell(\lambda^\star\mid Y)
\le1.
\]
Applying \eqref{eq:fixed-lr} with $u=1$ and taking a union bound over the
net points yields
\begin{align}
&\mathbb P_{\lambda^\star}
\left(
\inf_{\substack{\lambda\in\Lambda:\\
|\lambda-\lambda^\star|\ge r+\varepsilon_t}}
\left\{
\ell(\lambda\mid Y)
-
\ell(\lambda^\star\mid Y)
\right\}
\le0
\right)
\nonumber\\
&\qquad\le
e^{-t}
+
N_t
\exp\left\{
\frac12
-
c_\Lambda
\frac{dr^2}
{s_{\alpha,\Lambda}^2}
\right\}.
\tag{C.7}
\label{eq:uniform-lr}
\end{align}
This controls the likelihood ratio uniformly over the interval.

\subsection{MLE localization}
\label{app:slc-mle}

The uniform bound \eqref{eq:uniform-lr} gives the MLE risk bound in
Theorem~\ref{thm:slc}. Recall that
\[
\widehat\lambda_{\mathrm{MLE},P}(Y)
\in
\arg\min_{\lambda\in\Lambda}
\ell(\lambda\mid Y).
\]
Because both
$\widehat\lambda_{\mathrm{MLE},P}$ and $\lambda^\star$ belong to
$\Lambda$,
\[
\left|
\widehat\lambda_{\mathrm{MLE},P}
-
\lambda^\star
\right|^2
\le
\Delta_\Lambda^2.
\]
It remains to consider the regime in which the risk bound is nontrivial:
\[
\frac{
s_{\alpha,\Lambda}^2\log(2d)
}{d}
\le
\Delta_\Lambda^2.
\tag{C.8}
\label{eq:nontrivial-regime}
\]

Take
\[
t:=3\log(2d).
\]
By \eqref{eq:net-size},
\[
N_t
\le
C_\Lambda
s_{\alpha,\Lambda}(d+t).
\]
Under \eqref{eq:nontrivial-regime},
\[
\log N_t
\lesssim_\Lambda
\log(2d).
\]
Choose
\[
r_t
:=
C_\Lambda
s_{\alpha,\Lambda}
\sqrt{
\frac{\log(2d)}{d}
},
\]
with $C_\Lambda$ sufficiently large. Then
\eqref{eq:uniform-lr} gives
\[
\mathbb P_{\lambda^\star}
\left(
\inf_{\substack{\lambda\in\Lambda:\\
|\lambda-\lambda^\star|
\ge r_t+\varepsilon_t}}
\left\{
\ell(\lambda\mid Y)
-
\ell(\lambda^\star\mid Y)
\right\}
\le0
\right)
\le
2(2d)^{-3}.
\tag{C.9}
\label{eq:mle-good-event}
\]

By definition of the constrained MLE,
\[
\ell(
\widehat\lambda_{\mathrm{MLE},P}
\mid Y
)
\le
\ell(\lambda^\star\mid Y).
\]
Therefore, on the complement of the event in
\eqref{eq:mle-good-event},
\[
\left|
\widehat\lambda_{\mathrm{MLE},P}
-
\lambda^\star
\right|
\le
r_t+\varepsilon_t.
\]
Since
\[
\varepsilon_t
=
\frac1{
C_\Lambda
s_{\alpha,\Lambda}(d+t)
}
\]
and $s_{\alpha,\Lambda}\gtrsim_\Lambda1$, we have
\[
\varepsilon_t
\lesssim_\Lambda
s_{\alpha,\Lambda}
\sqrt{
\frac{\log(2d)}{d}
}.
\]
Hence, with probability at least
$1-2(2d)^{-3}$,
\[
\left|
\widehat\lambda_{\mathrm{MLE},P}
-
\lambda^\star
\right|
\lesssim_\Lambda
s_{\alpha,\Lambda}
\sqrt{
\frac{\log(2d)}{d}
}.
\tag{C.10}
\label{eq:mle-localization}
\]

Using the deterministic bound
$|\widehat\lambda_{\mathrm{MLE},P}-\lambda^\star|
\le\Delta_\Lambda$
on the complementary event, we obtain
\begin{align*}
\mathbb E_{\lambda^\star,P}
\left[
\left(
\widehat\lambda_{\mathrm{MLE},P}
-
\lambda^\star
\right)^2
\right]
&\lesssim_\Lambda
\frac{
s_{\alpha,\Lambda}^2\log(2d)
}{d}
+
\Delta_\Lambda^2(2d)^{-3} \\
&\lesssim_\Lambda
\frac{
s_{\alpha,\Lambda}^2\log(2d)
}{d}.
\end{align*}
The estimate is uniform over
$\lambda^\star\in\Lambda$. Since
$s_{\alpha,\Lambda}\asymp_\Lambda1+\alpha^{-1}$,
\[
\sup_{\lambda^\star\in\Lambda}
\mathbb E_{\lambda^\star,P}
\left[
\left(
\widehat\lambda_{\mathrm{MLE},P}
-
\lambda^\star
\right)^2
\right]
\lesssim_\Lambda
\frac{
(1+\alpha^{-1})^2\log(2d)
}{d}
\]
in the nontrivial regime. Combining this with the deterministic diameter
bound gives, for every $d\ge1$ and $\alpha>0$,
\[
\sup_{\lambda^\star\in\Lambda}
\mathbb E_{\lambda^\star,P}
\left[
\left(
\widehat\lambda_{\mathrm{MLE},P}
-
\lambda^\star
\right)^2
\right]
\lesssim_\Lambda
\min\left\{
\Delta_\Lambda^2,\,
\frac{
(1+\alpha^{-1})^2\log(2d)
}{d}
\right\}.
\]
This proves the MLE statement in Theorem~\ref{thm:slc}.

\end{document}